\documentclass[10pt]{amsart}

\newcommand{\Oplus}{\ensuremath{\vcenter{\hbox{\scalebox{1.5}{$\oplus$}}}}}

\newcounter{mnote}
\let\oldmarginpar\marginpar
\renewcommand\marginpar[1]{\-\oldmarginpar[\raggedleft\footnotesize #1]%
	{\raggedright\footnotesize #1}}

\usepackage[english]{babel}

\usepackage{amssymb}
\usepackage[leqno]{amsmath}
\usepackage{amsthm}
\usepackage{amsmath,amscd}
\usepackage{mathrsfs}
\usepackage{stmaryrd}
\usepackage{chemarrow}
\usepackage[norelsize]{algorithm2e}
\usepackage{enumerate}
\usepackage{graphicx}
\usepackage[all]{xy}
\usepackage{tikz}
\usetikzlibrary{arrows}
\usepackage{multirow}
\usepackage[colorinlistoftodos]{todonotes}
\usepackage[colorlinks=true, allcolors=blue]{hyperref}
\usepackage[hyperpageref]{backref}
\usepackage{xcolor,graphicx}
\usepackage{tabularx,array}
\usepackage{longtable}
\usepackage{booktabs,multirow,bigstrut,bigdelim,cprotect}
\usepackage{listings}
\usepackage{mdwlist}
\usepackage{mathtools}
\usepackage[section]{placeins}
\providecommand{\keywords}[1]
{
\small
\textbf{Keywords:}
}
\newtheorem{theorem}{Theorem}[section]
\newtheorem{lemma}[theorem]{Lemma}

\newtheorem{remark}[theorem]{Remark}

\usepackage{xparse}
\NewDocumentCommand{\xrightleftarrow}{ O{}O{} }{%
	\mathrel{%
		\vcenter{
			\hbox{%
				\begin{tikzpicture}
					\node[minimum width=0.75cm,align=center] (a){\text{\vphantom{hg}#1}\\[0.3ex]\vphantom{hg}#2};
					\draw[-stealth] ([yshift=0.4ex]a.west) -- ([yshift=0.4ex]a.east);
					\draw[stealth-] ([yshift=-0.4ex]a.west) -- ([yshift=-0.4ex]a.east);
				\end{tikzpicture}
			}
		}%
	}\!\!%
}

\newcommand{\curl}{\operatorname{curl}}
\renewcommand{\div}{\operatorname{div}}
\newcommand{\grad}{\operatorname{grad}}
\newcommand{\tr}{\operatorname{tr}}

\newcommand{\sym}{\operatorname{sym}}
\newcommand{\skw}{\operatorname{skw}}

\newcommand{\air}{\operatorname{airy}}
\numberwithin{equation}{section}

\begin{document}
	\title[Low-order FEM for linear elasticity]{New low-order mixed finite element methods for linear elasticity}
	\author{Xuehai Huang}%
	\address{School of Mathematics, Shanghai University of Finance and Economics, Shanghai 200433, China}%
	\email{huang.xuehai@sufe.edu.cn}%
	\author{Chao Zhang}%
	\address{School of Mathematics, Shanghai University of Finance and Economics, Shanghai 200433, China}%
	\email{zcmath@163.sufe.edu.cn}%
	\author{Yaqian Zhou}%
	\address{School of Mathematics, Shanghai University of Finance and Economics, Shanghai 200433, China}%
	\email{cxmvxbzhou@gmail.com}%
	\author{Yangxing Zhu}%
	\address{Network and Information Technology Center, Shanghai University of Finance and Economics, Shanghai 200433, China}%
	\email{zhu.yangxing@sufe.edu.cn}%
	
	\thanks{The first author was supported by the National Natural Science Foundation of China Project 12171300, and the Natural Science Foundation of Shanghai 21ZR1480500.}
	
\makeatletter
\@namedef{subjclassname@2020}{\textup{2020} Mathematics Subject Classification}
\makeatother

\begin{abstract}
New low-order ${H}(\div)$-conforming finite elements for symmetric tensors are constructed in arbitrary dimension. The space of shape functions is defined by enriching the symmetric quadratic polynomial space with the $(d+1)$-order normal-normal face bubble space. The reduced counterpart has only $d(d+1)^2$ degrees of freedom. Basis functions are explicitly given in terms of barycentric coordinates. Low-order conforming finite element elasticity complexes starting from the Bell element, are developed in two dimensions.
These finite elements for symmetric tensors are applied to devise robust mixed finite element methods for the linear elasticity problem, which possess the uniform error estimates with respect to the Lam\'{e} coefficient $\lambda$, and superconvergence for the displacement. Numerical results are provided to verify the theoretical convergence rates. 
\end{abstract}


	\maketitle
	
	
	
	\section{Introduction}
Based on the  Hellinger-Reissner variational principle, the mixed formulation of the linear elasticity employs the spaces ${H}(\div,\Omega; \mathbb{S})$ and ${L}^2(\Omega; \mathbb{R}^d)$ for the stress tensor and the displacement respectively.
However, it is difficult to develop stable mixed finite element methods using pure polynomials as shape functions for the symmetric stress tensor. One way to overcome this difficulty is to adopt composite elements~\cite{claes1978,ArnoldJimGupta1984}. Another approach is to enforce
the symmetry of the stress weakly by introducing a Lagrange multiplier \cite{AmaraThomas1979,ArnoldRichardRagnar2007Mixed,FortinBrezziBoffi2017Reduced,BernardoGopalakrishnanGuzman2010,FarhloulFortin997,GopalakrishnanGuzman2012,QiuDemkowicz2009}.
	
The first ${H}(\div)$-conforming finite element in two dimensions for symmetric tensors with pure polynomials as shape functions has been proposed in \cite{ArnoldWinthe2002}, which has been extended to higher dimensions in \cite{HuZhang2016lower,ArnoldAwanouWinther2007,AdamsCockburn2005}.
For this family of finite elements, the shape function is a $\mathbb{P}_{k+d-1}$ symmetric tensor whose divergence is a $\mathbb{P}_{k-1}$ vector on each simplex, where $d$ is the dimension of the geometric domain.
Hu and Zhang \cite{HuZhang2015,HuZhang2014A,Hu2015Higher} used the $k$th order polynomials as the shape functions to construct ${H}(\div)$-conforming finite elements for symmetric tensors with constraint $k\geq d+1$. By moving all the tangential-normal degrees of freedom (DoFs) of the Hu-Zhang element to the faces, Chen and Huang \cite{ChenHuang2022} designed a different ${H}(\div)$-conforming finite element for symmetric tensors, in which the constraint $k\geq d+1$ is also required.
On rectangular grids, we refer to \cite{ArnoldAwanou2005,Hu2015,HuManZhang2014,Awanou2012,ChenWang2011} for ${H}(\div)$-conforming finite element for symmetric tensors. 
The normal-normal continuous finite elements for symmetric tensors haven been proposed in \cite{ChenHuang2023,Hellan1967,Herrmann1967,Johnson1973,Sinwel2009,PechsteinSchoeberl2011}, where the
tangential-normal continuity is imposed weakly through the discrete bilinear form.
In addition, nonconforming finite elements for symmetric tensors were designed in \cite{ArnoldWinther2003,Awanou2009,ArnoldAwanouWinther2014,GopalakrishnanGuzman2011,XieXu2011,WuGongXu2017Interior,HuShi2007,ManHuShi2009,Yi2005,Yi2006}.

Many low-order mixed finite element methods with exactly symmetric stress have been devised for linear elasticity \cite{HuZhang2016lower,HuShi2007,ManHuShi2009,HuMaSun2023,CaiYe2005,ChenHuHuang2017}.
In three dimensions, when the stress is approximated by piecewise cubic polynomials and the displacement by discontinuous piecewise quadratic polynomials, the discrete inf-sup condition  was proved in \cite{HuMaSun2023} under some assumption on the mesh.
By removing the supersmooth DoFs of the finite elements in \cite{HuZhang2015,HuZhang2014A,Hu2015Higher}, a hybridized mixed method for linear elasticity was developed in~\cite{GongWuXu2019} for $k\geq d-1$ with $d=2,3$, where stability holds for $d-1\leq k\leq d$ on some special grids.

In this paper,  we shall construct low-order $H(\div)$-conforming finite elements for symmetric tensors in arbitrary dimension. By the DoFs (4.10)-(4.13) in \cite{ChenHuang2022}, the constraint $k\geq2$ is sufficient to ensure that $\div\mathbb P_k(K;\mathbb S)$ covers $\mathbb P_{k-1}(K;\mathbb R^d)/\mathbf{RM}$ and on face $F$ the tangential-normal part of $\mathbb P_k(K;\mathbb S)$ covers the tangential part of $\mathbf{RM}$, where $\mathbf{RM}$ is the space of rigid motions. The constraint $k\geq d+1$ is caused by the normal part of $\mathbf{RM}$ and the supersmooth on lower dimensional faces induced by the symmetry. To this end, we enrich $\mathbb P_{2}(K;\mathbb S)$ with the $(d+1)$-order normal-normal face bubble space $\mathbb B_{\partial K}^{nn}\subset\mathbb P_{d+1}(K;\mathbb S)$ to define the shape function space
$
\boldsymbol{\Sigma}(K):=\mathbb P_{2}(K;\mathbb S)\oplus\mathbb B_{\partial K}^{nn}.
$
The DoFs are chosen in the consideration of both $\mathbb P_{2}(K;\mathbb S)$ and $\mathbb B_{\partial K}^{nn}$:
\begin{align*}
\boldsymbol \tau (\texttt{v}) & \quad\forall~\texttt{v}\in \Delta_{0}(K), \\
(\boldsymbol  n_i^{\intercal}\boldsymbol \tau\boldsymbol n_j, q)_e & \quad\forall~q\in\mathbb P_{0}(e),  e\in\Delta_{1}(K), 1\leq i\leq j\leq d-1,\;  \\
(\boldsymbol  n^{\intercal}\boldsymbol \tau\boldsymbol n, q)_F & \quad\forall~q\in\mathbb P_{1}(F),  F\in\partial K,\;  \\
(\Pi_F\boldsymbol \tau\boldsymbol n, \boldsymbol q)_F & \quad\forall~\boldsymbol q\in {\rm ND}_{0}(F),  F\in\partial K, \\
 (\boldsymbol \tau, \boldsymbol q)_K &\;\;\;\,\forall~\boldsymbol q\in \mathbb P_{0}(K;\mathbb S),
\end{align*}
where ${\rm ND}_{0}(F)$ is the rigid motion space on face $F$.

To lower the dimension of the shape function space furthermore, we take a subspace $\boldsymbol{\Sigma}^r(K)$ of $\boldsymbol{\Sigma}(K)$, which consists of all the symmetric tensors in $\boldsymbol{\Sigma}(K)$ satisfying that: (i) its divergence belongs to $\mathbf{RM}$; (ii) the normal-normal part restricted to each edge is linear. The DoFs are listed as follows:
\begin{align*}
\boldsymbol \tau (\texttt{v}) & \quad\forall~\texttt{v}\in \Delta_{0}(K), \\
(\boldsymbol  n^{\intercal}\boldsymbol \tau\boldsymbol n, q)_F & \quad\forall~q\in\mathbb P_{1}(F),  F\in\partial K,\;  \\
(\Pi_F\boldsymbol \tau\boldsymbol n, \boldsymbol q)_F & \quad\forall~\boldsymbol q\in {\rm ND}_{0}(F),  F\in\partial K.
\end{align*}
The dimension of the reduced space $\boldsymbol{\Sigma}^r(K)$ is only $d(d+1)^2$.
The local dimension of the first-order $H(\div)$-conforming finite element for symmetric tensors in \cite{HuZhang2016lower} is also $d(d+1)^2$. Their finite element space is defined by enriching the symmetric tensor-valued linear element space with both the $(d+1)$-order normal-normal face bubble space and the $(d+1)$-order tangential-normal face bubble space, while our finite element space is enriched by the $(d+1)$-order normal-normal face bubble space and the second order tangential-normal face bubble space. Thereby, our reduced finite element for symmetric tensors is simpler than the first order one in \cite{HuZhang2016lower} in the sense that on each face $F\in\partial K$, $(\Pi_F\boldsymbol \tau\boldsymbol n)|_F\in\mathbb P_2(F;\mathbb R^{d-1})$ for $\boldsymbol{\tau}\in\boldsymbol{\Sigma}^r(K)$, while $(\Pi_F\boldsymbol \tau\boldsymbol n)|_F\in\mathbb P_{d+1}(F;\mathbb R^{d-1})$ for the first order symmetric tensor $\boldsymbol \tau$ in \cite{HuZhang2016lower}.
We present the explicit expressions of the basis functions of $\boldsymbol{\Sigma}(K)$ and $\boldsymbol{\Sigma}^r(K)$ in terms of barycentric coordinates.

We also present low-order conforming finite element discretizations of the
elasticity complex in two dimensions \cite{Eastwood2000}
	\begin{align}\label{stresscomplex}
		\mathbb{P}_1(\Omega)\xrightarrow{\subset}H^2(\Omega) \xrightarrow{\air} H(\div,\Omega; \mathbb{S}) \xrightarrow{\div}L^2(\Omega; \mathbb{R}^2) \rightarrow 0,
	\end{align}
	where $\air$ is the Airy operator.
The Bell element \cite{Bell1969} is adopted to discretize the space $H^2(\Omega)$, the space $H(\div,\Omega; \mathbb{S})$ is discretized by the global version of $\boldsymbol{\Sigma}(K)$ or $\boldsymbol{\Sigma}^r(K)$, and the space  $L^2(\Omega; \mathbb{R}^2)$ is discretized by the piecewise linear element or the piecewise rigid motion.

With the previously constructed ${H}(\div; \mathbb{S})$-conforming finite elements for stress and the piecewise linear element space or the piecewise rigid motion space for displacement, we propose robust mixed finite element methods for the linear elasticity problem. After establishing two discrete inf-sup conditions in different norms and the interpolation error estimates, we achieve the optimal convergence for stress and some superconvergent estimates for displacement, all of which are robust with respect to the Lam\'{e} coefficient $\lambda$.

The rest of this paper is organized as follows. Low-order finite element elasticity complexes in two dimensions are developed in Section~\ref{sec:femelascomplex}. 
Low-order ${H}(\div; \mathbb{S})$-conforming finite elements in higher dimensions are constructed in Section~\ref{sec:divSfem}.
In Section \ref{sec:mfemelas}, we propose and analyze low-order mixed finite element methods for linear elasticity.  Finally, numerical results are presented in Section~\ref{sec:numresult}.

\section{Low-order finite element elasticity complexes in two dimensions}\label{sec:femelascomplex}	

We will construct two low-order finite element elasticity complexes on triangular meshes in two dimensions in this section.
\subsection{Notation}	
Let $\Omega\subset \mathbb{R}^d~(d\geq 2)$ be a bounded
polytope.
	Given a bounded domain $D \subset \mathbb{R}^{d}$ with the diameter $h_D$,  we use $H^{m}(D) (m\geq 0)$ to denote the set of all $L^{2}(D)$ functions whose derivatives up to order $m$ are also square-integrable. Set $\mathbb M:=\mathbb R^{d\times d}$.
Denote by $\mathbb S$ and $\mathbb K$ the subspace of symmetric matrices and skew-symmetric matrices of $\mathbb M$, respectively. For a space $B(D)$ defined on $D$,
let $B(D; \mathbb{X})=B(D)\otimes\mathbb{X}$ be its vector or tensor version for $\mathbb{X}$ being $\mathbb{R}^d$, $\mathbb{M}$, $\mathbb{S}$ and $\mathbb{K}$. Set $B(D; \mathbb{R})=B(D)$ for simplicity. The norm and semi-norm of ${H}^{m}(D; \mathbb{X})$ are denoted, respectively, by $\|\cdot \|_{m, D}$ and $|\cdot |_{m, D}$. Let $(\cdot, \cdot)_{D}$ be the standard inner product on ${L}^{2}(D; \mathbb{X})$.  Denote by ${H}^{m}_{0}(D; \mathbb{X})$ the closure of ${C}^{\infty}_{0}(D; \mathbb{X})$ with respect to the norm $\|\cdot \|_{m, D}$. Let ${H}(\operatorname{div}, D; \mathbb{S})$ consist of square-integrable symmetric
	matrix fields with square-integrable divergence. The ${H}(\div)$ norm is defined by
	$$
	\|\boldsymbol{\tau}\|_{{H}(\operatorname{div}, D)}^2:=\|\boldsymbol{\tau}\|_{0, D}^2+\|\operatorname{div} \boldsymbol{\tau}\|_{0, D}^2.
	$$
Let $H_0(\div, D;\mathbb{S})$ be the closure of ${C}^{\infty}_{0}(D; \mathbb{S})$ with respect to the norm $\|\cdot\|_{{H}(\operatorname{div}, D)}$.
When $D=\Omega$, we will abbreviate $\|\cdot \|_{m, \Omega}$, $|\cdot |_{m, \Omega}$, $(\cdot, \cdot)_{\Omega}$ and $\|\cdot\|_{{H}(\operatorname{div}, \Omega)}$ as $\|\cdot \|_{m}$, $|\cdot |_{m}$, $(\cdot, \cdot)$ and $\|\cdot\|_{{H}(\operatorname{div})}$, respectively.

For a bounded domain $D\subset\mathbb{R}^{d}$ and 
a non-negative integer $k$, 
let $\mathbb P_k(D)$ stand for the set of all polynomials in $D$ with total degree no more than $k$. When $k<0$, $\mathbb P_k(D) = \{0\}.$
Let $\mathbb H_k(D):=\mathbb P_k(D)\backslash\mathbb P_{k-1}(D)$ be the space of homogeneous polynomials of degree $k$. 
Recall that 
$$
\dim \mathbb P_{k}(D) = { k + d \choose d} = { k + d \choose k},\quad \dim \mathbb H_{k}(D) = { k + d -1  \choose d-1} = { k + d - 1 \choose k} 
$$ for a $d$-dimensional domain $D$.
Let $Q_{D}^{k} : L^{2}(D) \rightarrow \mathbb{P}_{k}(D)$ be the $L^{2}$-orthogonal projector, and its vector version is also denoted by $Q_{D}^{k}$. Set $Q_{D} := Q_{D}^{0}$. 
For an integer $k\geq 0$, the shape function space of the first kind Ned\'el\'ec element~\cite{Nedelec1986} is
$$
{\rm ND}_{k}(D) = \mathbb P_{k}(K;\mathbb R^d) \oplus \mathbb H_{k}(D; \mathbb K)\boldsymbol x =\grad \mathbb P_{k+1}(D)\oplus\mathbb P_{k}(D;\mathbb K)\boldsymbol x
$$
with $\mathbb P_{k}(D;\mathbb K)\boldsymbol x:=\{\boldsymbol{\tau}\boldsymbol{x}: \boldsymbol{\tau}\in\mathbb P_{k}(D;\mathbb K)\}$.
Let ${\bf RM} :={\rm ND}_{0}(D)$, which is the space of rigid motions.

For a $d$-dimensional simplex $K$, we let $\Delta(K)$ denote all the subsimplices of $K$, while $\Delta_{\ell}(K)$ denotes the set of subsimplices of dimension $\ell$, for $0\leq \ell \leq d$. Elements of $\Delta_0(K) = \{\texttt{v}_0, \texttt{v}_1, \ldots, \texttt{v}_d\}$ are $d+1$ vertices of $K$ and $\Delta_d(K) = K$.
Let $\lambda_i$ be the barycentric coordinates with respect to the vertex $\texttt{v}_i$. 
Set $\boldsymbol{t}_{ij}:=\texttt{v}_j-\texttt{v}_i \text{ for } 0\leq i, j \leq d$. We have \cite[Section 5.2]{ChenHuang2022a}
	\begin{equation}
		\boldsymbol{t}_{ij} \cdot\nabla \lambda_{\ell}=
		\delta_{j{\ell}}-\delta_{i{\ell}}=
		\begin{cases}
			1, & \text{if} \ {\ell}=j, \\
			-1,& \text{if} \ {\ell}=i, \\
			0, & \text{if} \ {\ell} \neq i,j.
		\end{cases}
		\label{tij}
	\end{equation}

For $f\in\Delta_{\ell}(K)$ with $0\leq \ell\leq d$, let $\boldsymbol n_{f,1}, \cdots, \boldsymbol n_{f,d-\ell}$ be its
mutually perpendicular unit normal vectors, and $\boldsymbol t_{f,1}, \cdots, \boldsymbol t_{f,\ell}$ be its
mutually perpendicular unit tangential vectors.
We abbreviate $\boldsymbol n_{f,1}$ as $\boldsymbol n_{f}$ or $\boldsymbol n$ when $\ell=d-1$, and $\boldsymbol t_{f,1}$ as $\boldsymbol t_{f}$ or $\boldsymbol t$ when $\ell=1$.
We also abbreviate $\boldsymbol n_{f,i}$ and $\boldsymbol t_{f,i}$ as $\boldsymbol n_{i}$ and $\boldsymbol t_{i}$ respectively if not causing any confusion. 
Furthermore, we use $\boldsymbol n_{\partial K}$ to denote the unit outward normal vector of $\partial K$, which will be abbreviated as $\boldsymbol{n}$ if not causing any confusion. 
Given a face $F\in\Delta_{d-1}(K)$, and a vector $\boldsymbol v\in \mathbb R^d$, define 
$$
\Pi_F\boldsymbol v= (\boldsymbol n_F\times \boldsymbol v)\times \boldsymbol n_F = (\boldsymbol I - \boldsymbol n_F\boldsymbol n_F^{\intercal})\boldsymbol v
$$ 
as the projection of $\boldsymbol v$ onto the face $F$. Let $b_F\in\mathbb P_d(K)$ be the face bubble function of $F$, then $b_F=\lambda_0\cdots\lambda_{i-1}\lambda_{i+1}\cdots\lambda_{d}$ when $F$ is opposite to vertex $\texttt{v}_i$. 
For edge $e\in\Delta_1(K)$ having end points $\texttt{v}_i$ and $\texttt{v}_j$, define the edge bubble function $b_e=\lambda_i\lambda_j$.


	
Denote by $\mathcal{T}_h$ a conforming triangulation of $\Omega$ with each geometric element being a simplex, where $h=\max_{T\in\mathcal{T}_h}h_T$. Let $\Delta_{\ell}(\mathcal{T}_h)$ denotes the set of all $\ell$-dimensional subsimplices of $\mathcal{T}_h$ for $0\leq \ell \leq d-1$.	
Let $H^1(\mathcal{T}_h;\mathbb X):=H^1(\mathcal{T}_h)\otimes \mathbb X$, where
$$
H^1(\mathcal{T}_h):=\{v\in L^2(\Omega): v|_K\in H^1(K)\;\; \textrm{ for } K\in\mathcal{T}_h\}.
$$	

	
	Consider two adjacent simplices $K^+$ and $K^-$ sharing an interior $(d-1)$-dimensional face $F$. Denote
by $\boldsymbol{n}^+$ and $\boldsymbol{n}^-$ the unit outward normals to the common face $F$ of the simplices
$K^+$ and $K^-$, respectively. For a scalar-valued or vector-valued function $v$, write $v^+ := v|_{K^+}$ and
$v^- := v|_{K^-}$. Then define the jump of $v$ on $F$ as follows:
$$
\llbracket v\rrbracket:=v^{+} \boldsymbol{n}_{F} \cdot \boldsymbol{n}^{+}+v^{-} \boldsymbol{n}_{F} \cdot \boldsymbol{n}^{-}.
$$
On a face $F \in\partial K$ lying on the boundary $\partial\Omega$, the above term is defined by
$\llbracket v\rrbracket:=v\boldsymbol{n}_{F} \cdot \boldsymbol{n}_{\partial K}.$

For vectors $\boldsymbol{u}$ and $\boldsymbol{v}$, define the tensor product $\boldsymbol{u}\otimes\boldsymbol{v}=\boldsymbol{u}\boldsymbol{v}^{\intercal}.$
Let the symmetric part $\sym\boldsymbol{\boldsymbol{\tau}}:=\frac{1}{2}(\boldsymbol{\tau}+\boldsymbol{\tau}^{\intercal})$ for tensor function $\boldsymbol{\tau}$.
Define the symmetric gradient $\boldsymbol{\varepsilon}(\boldsymbol{v}):=\sym(\nabla\boldsymbol{v})$ for smooth vector function $\boldsymbol{v}$, then ${\bf RM}$ is the kernel of the strain operator $\boldsymbol{\varepsilon}$.
For scalar function $v$ in two dimensions, denote
$$
\curl v:=\left(
\frac{\partial v}{\partial x_2},
-\frac{\partial v}{\partial x_1}
\right)^{\intercal},
\quad
\air v:={\curl}\curl v=\begin{pmatrix}
\medskip
\frac{\partial^2v}{\partial x_2^2} & -\frac{\partial^2v}{\partial x_1\partial x_2} \\
-\frac{\partial^2v}{\partial x_1\partial x_2} & \frac{\partial^2v}{\partial x_1^2}
\end{pmatrix}.
$$
	
	
	Throughout this paper, we use
	``$\lesssim\cdots $" to mean that ``$\leq C\cdots$", where
	letter $C$ is a generic positive constant independent of $h$ and the parameter $\lambda$,
	which may stand for different values at its different occurrences. And the notation $A\eqsim B$ means $B\lesssim A\lesssim B$. Denote by $\#A$ the number of elements in a finite set $A$.

\subsection{A low-order ${H}(\div; \mathbb{S})$-conforming finite element}	
We focus on constructing a conforming low-order
finite element for the space ${H}(\operatorname{div}, \Omega; \mathbb{S})$ in two dimensions in this subsection.  To this end, recall
the double-directional polynomial complex \cite[Section 3]{ChenHuang2020}: for integer $k\geq 0$, the sequence
	\begin{align*}
	0\rightarrow\mathbb{P}_1(K)\xrightleftarrow[$\subset$][$\pi_1$]\mathbb{P}_{k+2}(K)\xrightleftarrow[$\air$][$(\boldsymbol{x}^{\perp})^{\intercal}\boldsymbol{\tau}\boldsymbol{x}^{\perp}$] \mathbb{P}_{k}(K,\mathbb{S}) \xrightleftarrow[$\div$][$\sym(\boldsymbol{v}\boldsymbol{x}^\intercal)$] \mathbb{P}_{k-1}(K; \mathbb{R}^2) \xrightleftarrow[][$\supset$] 0
\end{align*}
is exact, where $\boldsymbol{x}^{\perp}:=(x_2,-x_1)^{\intercal}$ is the rotation of $\boldsymbol{x}$, and $\pi_1v:=v(0,0)+\boldsymbol{x}^{\intercal}\nabla v(0,0)$. 
By this double-directional polynomial complex, we have the decomposition
\begin{equation}\label{eq:Spolydecomp}	
\mathbb{P}_{k}(K,\mathbb{S}):=\air\mathbb{P}_{k+2}(K) \oplus \sym(\mathbb{P}_{k-1}(K; \mathbb{R}^2)\boldsymbol{x}^\intercal),
\end{equation}
where $\mathbb{P}_{k}(K; \mathbb{R}^2)\boldsymbol{x}^\intercal:=\{\boldsymbol{v}\boldsymbol{x}^\intercal: \boldsymbol{v}\in\mathbb{P}_{k}(K; \mathbb{R}^2)\}$.

Notice that for a rigid motion $\boldsymbol{v}\in\mathbf{RM}$, the tangential part $\boldsymbol{v}\cdot\boldsymbol{t}$ on edge $e$ is a constant rather than a linear polynomial. This fact and the decomposition \eqref{eq:Spolydecomp} motivate us to take the space of shape functions on triangle $K \in \mathcal{T}_h$ as
\begin{equation*}
\boldsymbol{\Sigma}(K):=\{ \boldsymbol{\tau} \in \mathbb{P}_3(K;\mathbb{S}):\div \boldsymbol{\tau} \in \mathbb{P}_1(K;\mathbb{R}^2), \boldsymbol{t}^{\intercal} \boldsymbol{\tau}\boldsymbol{n} |_e \in \mathbb{P}_2(e) \ \ \forall~e \in \Delta_1(K) \}.
\end{equation*}

\begin{lemma}
For $K \in \mathcal{T}_h$, we have
$$
\boldsymbol{\Sigma}(K)=\air\mathbb{P}_5^-(K)\oplus \sym(\mathbb{P}_1(K;\mathbb{R}^2)\boldsymbol{x}^\intercal),
$$
where the shape function space of Bell element \cite{Bell1969}
\begin{equation*}
\mathbb{P}_5^-(K):=\left \{ 
q \in \mathbb{P}_5(K):\partial _n q|_e \in \mathbb{P}_3(e) \quad \forall~e \in \Delta_1(K)
\right \}.
\end{equation*}
\end{lemma}
\begin{proof}
It is easy to see that $\air\mathbb{P}_5^-(K)\oplus \sym(\mathbb{P}_1(K;\mathbb{R}^2)\boldsymbol{x}^\intercal)\subseteq\boldsymbol{\Sigma}(K)$. To prove the other side, we introduce $\widetilde{\boldsymbol{\Sigma}}(K)= \{ \boldsymbol{\tau} \in \mathbb{P}_3(K;\mathbb{S}):\div \boldsymbol{\tau} \in \mathbb{P}_1(K;\mathbb{R}^2)\}$, then $\boldsymbol{\Sigma}(K)=\{\boldsymbol{\tau} \in \widetilde{\boldsymbol{\Sigma}}(K):\boldsymbol{t}^{\intercal} \boldsymbol{\tau}\boldsymbol{n} |_e \in \mathbb{P}_2(e) \ \ \forall~e \in \Delta_1(K) \}$. By the decomposition \eqref{eq:Spolydecomp}, it holds that
$$
\widetilde{\boldsymbol{\Sigma}}(K)=\air\mathbb{P}_5(K)\oplus \sym(\mathbb{P}_1(K;\mathbb{R}^2)\boldsymbol{x}^\intercal).
$$
Take $\boldsymbol{\tau}=\air q\oplus \sym(\boldsymbol{p}\boldsymbol{x}^\intercal)\in\boldsymbol{\Sigma}(K)$ with $q\in\mathbb{P}_5(K)$ and $\boldsymbol{p}\in\mathbb{P}_1(K;\mathbb{R}^2)$. Since $\boldsymbol{t}^{\intercal}\boldsymbol{\tau}\boldsymbol{n} |_e \in \mathbb{P}_2(e)$ for $e \in \Delta_1(K)$, we have $\boldsymbol{t}^{\intercal}(\air q)\boldsymbol{n} |_e \in \mathbb{P}_2(e)$, which means $\partial _n q|_e \in \mathbb{P}_3(e)$. Thus, $q\in\mathbb{P}_5^-(K)$, which ends the proof.
\end{proof}

We refer
	to \cite{Okabe1994} for the basis functions of Bell element. The dimension of $\boldsymbol{\Sigma}(K)$ is 21.
	The degrees of freedom (DoFs) for $\boldsymbol{\Sigma}(K)$ are given by
\begin{subequations}\label{sigmadof}
		\begin{align}
		\label{sigmadof1}	\boldsymbol{\tau}(\texttt{v}) & \quad \forall~\texttt{v} \in \Delta_0(K),\\
		\label{sigmadof2}	(\boldsymbol{n}^{\intercal}\boldsymbol{\tau} \boldsymbol{n},q)_e & \quad\forall~q\in \mathbb{P}_{1}(e), e\in \partial K,\\
		\label{sigmadof3}	(\boldsymbol{t}^{\intercal}\boldsymbol{\tau} \boldsymbol{n},q)_e & \quad\forall~q\in \mathbb{P}_{0}(e), e\in \partial K,\\
		\label{sigmadof4}	(\boldsymbol{\tau}, \boldsymbol{q})_K & \quad \forall~\boldsymbol{q} \in \mathbb{P}_{0}(K ; \mathbb{S}).
	\end{align}
\end{subequations}

	\begin{lemma}
		The DoFs \eqref{sigmadof} are unisolvent for $\boldsymbol{\Sigma}(K)$.
	\end{lemma}
	\begin{proof}
The number of DoFs \eqref{sigmadof} is 21, which is same as
		the dimension of $\boldsymbol{\Sigma}(K)$. Take any $\boldsymbol{\tau} \in \boldsymbol{\Sigma}(K)$ and suppose all the DoFs \eqref{sigmadof} vanish.
Since $\boldsymbol{n}^{\intercal}\boldsymbol{\tau} \boldsymbol{n}|_e \in \mathbb{P}_3(e)$ and $\boldsymbol{t}^{\intercal}\boldsymbol{\tau} \boldsymbol{n}|_e \in \mathbb{P}_2(e)$ for each $e \in \partial K$, we get from the vanishing DoFs (\ref{sigmadof1})-(\ref{sigmadof3}) that $\boldsymbol{\tau}\boldsymbol{n}|_{\partial K} = 0$. Applying the integration
by parts, we get from the vanishing DoF (\ref{sigmadof4}) that
		$$
		(\div \boldsymbol{\tau}, \boldsymbol{q})_K=(\boldsymbol{\tau}, \boldsymbol{\varepsilon}(\boldsymbol{q}))_K=0 \quad \forall~\boldsymbol{q} \in \mathbb{P}_1(K ; \mathbb{R}^2).
		$$
		Hence $\div\boldsymbol{\tau} = 0$. So there exists $q \in \mathbb{P}^-_5 (K)$ such that $\boldsymbol{\tau} = \air q$, and
$\partial_t(\nabla q)|_e = 0$ for each $e \in \Delta_1(K)$. Thus $(\nabla q)|_{\partial K}$ is constant and $q|_e$ is a linear function. Then we can assume $q|_{\partial K} = 0$. As a result, there exists $\tilde{q} \in \mathbb{P}_2(K)$ such that
		$q = b_K\tilde{q}$ with $b_K$ being the cubic bubble function on $K$. Noting that $\partial_n q|_e = (\tilde{q}\partial_n b_K)|_e$ is constant and $\partial_n b_K|_e \in \mathbb{P}_2(e)$, we get
		$\tilde{q}|_e = 0$. Finally, $\tilde{q} = q = 0$ and $\boldsymbol{\tau} = 0$.
	\end{proof}
	

	\subsection{Basis functions}\label{sub:basis}

\begin{lemma}\label{lem:phi}
For edge $e_i$ having end points ${\rm v}_j$ and ${\rm v}_k$, let
\begin{align}\label{phi_}
			\boldsymbol{\psi}_i :=
			&-10c_{ik} \lambda_j \lambda_k^2 \boldsymbol{t}_{ki} \otimes \boldsymbol{t}_{ki}-10c_{ij}(\lambda_k+\lambda_i)\lambda_j\lambda_k \boldsymbol{t}_{ij} \otimes \boldsymbol{t}_{ij} \\
			&\quad+[(6c_{ij}-9c_{ik})\lambda_j+(9c_{ij}-6c_{ik})\lambda_k]\lambda_j\lambda_k \boldsymbol{t}_{jk} \otimes \boldsymbol{t}_{jk} \notag\\
			&\quad+[(6c_{ij}-9c_{ik})\lambda_j+(2c_{ik}-3c_{ij})\lambda_k]\lambda_k\lambda_i \boldsymbol{t}_{ki} \otimes \boldsymbol{t}_{ki} \notag\\
			&\quad+[(3c_{ik}-2c_{ij})\lambda_j+(9c_{ij}-6c_{ik})\lambda_k]\lambda_i\lambda_j \boldsymbol{t}_{ij} \otimes \boldsymbol{t}_{ij}, \notag
		\end{align}
		where $c_{ij}=\nabla\lambda_i\cdot\nabla\lambda_j$ and $(i, j, k)$ is a cyclic permutation of $(0,1,2)$. 
We have $\boldsymbol{\psi}_i \in \boldsymbol{\Sigma}(K)$, and
		$$
(\boldsymbol{\psi}_i \boldsymbol{n})|_{e_j}=0,\;\; (\boldsymbol{\psi}_i \boldsymbol{n})|_{e_k}=0, \;\;(\boldsymbol{t}^{\intercal}\boldsymbol{\psi}_i \boldsymbol{n})|_{e_i}=0, \;\;
		(\boldsymbol{n}^{\intercal} \boldsymbol{\psi}_i \boldsymbol{n})|_{e_i}=(10\lambda_j \lambda_k^2)|_{e_i} .
		$$
	\end{lemma}
	\begin{proof}
It is easy to see that $(\boldsymbol{\psi}_i \boldsymbol{n})|_{e_j}=0$ and $(\boldsymbol{\psi}_i \boldsymbol{n})|_{e_k}=0$ follows from \eqref{tij}, $\lambda_j|_{e_j}=0$ and $\lambda_k|_{e_k}=0$.
Thanks to \eqref{tij} and the fact $\lambda_i|_{e_i}=0$,
$$
(\boldsymbol{\psi}_i \nabla \lambda_i)|_{e_i}=10(c_{ik} \lambda_j \lambda_k^2 \boldsymbol{t}_{ik}+ c_{ij} \lambda_j \lambda_k^2 \boldsymbol{t}_{ij})|_{e_i}.
$$
Noting that $\boldsymbol{t}_{jk}\cdot\boldsymbol{t}_{ki}=4|K|^2c_{ij}$ and $\boldsymbol{t}_{jk}\cdot\boldsymbol{t}_{ij}=4|K|^2c_{ik}$,
we have 
$$
(\boldsymbol{t}_{jk}^{\intercal}\boldsymbol{\psi}_i \nabla \lambda_i)|_{e_i}=40|K|^2(-c_{ij}c_{ik} \lambda_j \lambda_k^2 + c_{ij}c_{ik} \lambda_j \lambda_k^2)|_{e_i}=0,
$$ 
$$
((\nabla \lambda_i)^{\intercal}\boldsymbol{\psi}_i \nabla \lambda_i)|_{e_i}=-10(c_{ij}+c_{ik})(\lambda_j \lambda_k^2)|_{e_i}=10 |\nabla\lambda_i|^2(\lambda_j \lambda_k^2)|_{e_i}.
$$

Next we prove $\div\boldsymbol{\psi}_i\in\mathbb{P}_1(K;\mathbb{R}^2)$. 
Using the identity $\div(q\boldsymbol{t}\otimes\boldsymbol{t})=(\boldsymbol{t}\cdot\nabla q)\boldsymbol{t}$ and \eqref{tij},
\begin{align*}
\div\boldsymbol{\psi}_i =
&20c_{ik} \lambda_j \lambda_k\boldsymbol{t}_{ki}
-10c_{ij}(\lambda_k^2+\lambda_i\lambda_k-\lambda_j\lambda_k)\boldsymbol{t}_{ij} \\
&+(2c_{ij}-3c_{ik})[\boldsymbol{t}_{jk}(3\lambda_j^2-6\lambda_j\lambda_k)+\boldsymbol{t}_{ki}(3\lambda_j\lambda_k-3\lambda_i\lambda_j)-\boldsymbol{t}_{ij}(2\lambda_i\lambda_j-\lambda_j^2)]  \\
&+(3c_{ij}-2c_{ik})[\boldsymbol{t}_{jk}(6\lambda_j\lambda_k-3\lambda_k^2)-\boldsymbol{t}_{ki}(\lambda_k^2-2\lambda_k\lambda_i)+\boldsymbol{t}_{ij}(3\lambda_i\lambda_k-3\lambda_j\lambda_k)].
\end{align*}
Employing $\boldsymbol{t}_{jk}=-\boldsymbol{t}_{ij}-\boldsymbol{t}_{ki}$ and $\lambda_i+\lambda_j+\lambda_k=1$, we get
\begin{align*}
\div\boldsymbol{\psi}_i =
&20c_{ik} \lambda_j \lambda_k\boldsymbol{t}_{ki} -10c_{ij}(\lambda_k^2+\lambda_i\lambda_k-\lambda_j\lambda_k)\boldsymbol{t}_{ij} \\
&+(2c_{ij}-3c_{ik})[\boldsymbol{t}_{ki}(9\lambda_j\lambda_k-3\lambda_i\lambda_j-3\lambda_j^2)-\boldsymbol{t}_{ij}(2\lambda_j^2+2\lambda_i\lambda_j-6\lambda_j\lambda_k)]  \\
&+(3c_{ij}-2c_{ik})[-\boldsymbol{t}_{ki}(6\lambda_j\lambda_k-2\lambda_k^2-2\lambda_k\lambda_i)+\boldsymbol{t}_{ij}(3\lambda_i\lambda_k+3\lambda_k^2-9\lambda_j\lambda_k)], \\
=&20c_{ik} \lambda_j \lambda_k\boldsymbol{t}_{ki}-10c_{ij}(\lambda_k-2\lambda_j\lambda_k)\boldsymbol{t}_{ij} \\
&+(2c_{ij}-3c_{ik})[\boldsymbol{t}_{ki}(12\lambda_j\lambda_k-3\lambda_j)-\boldsymbol{t}_{ij}(2\lambda_j-8\lambda_j\lambda_k)]  \\
&+(3c_{ij}-2c_{ik})[-\boldsymbol{t}_{ki}(8\lambda_j\lambda_k-2\lambda_k)+\boldsymbol{t}_{ij}(3\lambda_k-12\lambda_j\lambda_k)]  \\
=&-10c_{ij}\lambda_k\boldsymbol{t}_{ij} +(2c_{ij}-3c_{ik})(-3\lambda_j\boldsymbol{t}_{ki}-2\lambda_j\boldsymbol{t}_{ij})
+(3c_{ij}-2c_{ik})(2\lambda_k\boldsymbol{t}_{ki}+3\lambda_k\boldsymbol{t}_{ij})  \\
& +4[5c_{ik}\boldsymbol{t}_{ki}+5c_{ij}\boldsymbol{t}_{ij}
+(2c_{ij}-3c_{ik})(3\boldsymbol{t}_{ki}+2\boldsymbol{t}_{ij})-(3c_{ij}-2c_{ik})(2\boldsymbol{t}_{ki}+3\boldsymbol{t}_{ij})]\lambda_j\lambda_k.
\end{align*}
Notice that $5c_{ik}\boldsymbol{t}_{ki}+5c_{ij}\boldsymbol{t}_{ij}
+(2c_{ij}-3c_{ik})(3\boldsymbol{t}_{ki}+2\boldsymbol{t}_{ij})-(3c_{ij}-2c_{ik})(2\boldsymbol{t}_{ki}+3\boldsymbol{t}_{ij})=0$. Therefore, $\div\boldsymbol{\psi}_i\in\mathbb{P}_1(K;\mathbb{R}^2)$ and $\boldsymbol{\psi}_i \in \boldsymbol{\Sigma}(K)$.
\end{proof}

For edge $e_i$ with ending points $\texttt{v}_j$ and $\texttt{v}_k$, let $\boldsymbol{n}_i$ and $\boldsymbol{t}_i$ be the unit normal vector and unit tangential vector respectively.
The basis functions of $\boldsymbol{\Sigma}(K)$ are given as follows:
	\begin{itemize}
		\item[(i)] For vertex $\texttt{v}_i$ with $i=0, 1, 2$, the basis functions related to DoF (\ref{sigmadof1}) are
		$$
		\lambda_i^2 \begin{pmatrix}
			1&0 \\ 
			0&0
		\end{pmatrix}, \quad
		\lambda_i^2\begin{pmatrix}
			0&0 \\ 
			0&1
		\end{pmatrix},\quad
		\lambda_i^2\begin{pmatrix}
			0&1 \\ 
			1&0
		\end{pmatrix}.
$$
		
	\item[(ii)] For edge $e_i$ having end points $\texttt{v}_j$ and $\texttt{v}_k$, the basis functions related to DoF~(\ref{sigmadof2}) are 
	$$\boldsymbol{\phi}_{i,1}=36\lambda_j \lambda_k\boldsymbol{n}_i\otimes\boldsymbol{n}_i-6\boldsymbol{\psi}_i, \quad \boldsymbol{\phi}_{i,2}=-24\lambda_j \lambda_k\boldsymbol{n}_i\otimes\boldsymbol{n}_i+6\boldsymbol{\psi}_i,$$ 
where $\boldsymbol{\psi}_i$ is given by \eqref{phi_}. We have
\begin{align*}
\frac{1}{|e_i|}\int_{e_i}(\boldsymbol{n}^{\intercal}\boldsymbol{\phi}_{i,1}\boldsymbol{n})\lambda_j\textrm{ds}=1, \quad \frac{1}{|e_i|}\int_{e_i}(\boldsymbol{n}^{\intercal}\boldsymbol{\phi}_{i,1}\boldsymbol{n})\lambda_k\textrm{ds}=0, \\
\frac{1}{|e_i|}\int_{e_i}(\boldsymbol{n}^{\intercal}\boldsymbol{\phi}_{i,2}\boldsymbol{n})\lambda_j\textrm{ds}=0, \quad \frac{1}{|e_i|}\int_{e_i}(\boldsymbol{n}^{\intercal}\boldsymbol{\phi}_{i,2}\boldsymbol{n})\lambda_k\textrm{ds}=1.
\end{align*}
\item[(iii)] For edge $e_i$ having end points $\texttt{v}_j$ and $\texttt{v}_k$, the basis functions related to DoF~(\ref{sigmadof3}) are $\lambda_j \lambda_k \sym(\boldsymbol{t}_i \otimes \boldsymbol{n}_i).$
		
		\item[(iv)] The basis functions related to DoF (\ref{sigmadof4}) are $\lambda_j\lambda_k \boldsymbol{t}_i \otimes \boldsymbol{t}_i.$
	\end{itemize}

With the help of Lemma \ref{lem:phi}, it is easy to verify that  functions defined in the above four terms
form a basis of $\boldsymbol{\Sigma}(K)$.

\subsection{Finite element elasticity complex}
We will consider a low-order conforming finite element discretization of the elasticity complex (\ref{stresscomplex}), in which the spaces $H^2(\Omega)$
and ${L}^2
(\Omega, \mathbb{R}^2)$ are discretized by the Bell element and the piecewise linear element, respectively. Recall the Bell element space \cite{Bell1969}
	\begin{equation}
		W_h:=\left \{ v_h \in H^2(\Omega):
		v_h|_K \in \mathbb{P}_5^-(K) \quad \forall~K\in \mathcal{T}_h \right \},
		\notag
	\end{equation}
and	the piecewise linear element space
	\begin{align*}
		&\boldsymbol{V}_h:=\{\boldsymbol{q}_h\in {L}^{2}(\Omega;\mathbb{R}^{2}): \boldsymbol{q}_h|_K\in\mathbb{P}_{1}(K ; \mathbb{R}^{2})\quad \forall~K \in \mathcal{T}_{h}\}.
	\end{align*}
The DoFs for Bell element are
\begin{align*}
v(\texttt{v}), \nabla v(\texttt{v}), \nabla^2v(\texttt{v}) & \quad \forall~\texttt{v} \in \Delta_0(K).
\end{align*}
Define the global ${H}(\div,\Omega; \mathbb{S})$-conforming element space
$$
\boldsymbol{\Sigma}_h:=\{\boldsymbol{\tau}_h \in L^2(\Omega; \mathbb{S}): \boldsymbol{\tau}_h|_K\in \boldsymbol{\Sigma}(K)\;\; \forall\,K \in \mathcal{T}_{h}, \text{ DoFs \eqref{sigmadof} are single-valued}\}. 
$$
Thanks to the single-valued DoFs \eqref{sigmadof}, 
$\boldsymbol{\Sigma}_h\subset{H}(\div,\Omega ; \mathbb{S})$.

We combine these finite element spaces to form the discrete elasticity complex
\begin{equation}\label{disstresscomplex}
\mathbb{P}_1(\Omega)\xrightarrow{\subset}W_h \xrightarrow{\air}\boldsymbol{\Sigma}_h  \xrightarrow{\div}\boldsymbol{V}_h \rightarrow 0.
\end{equation}
To prove the exactness of the discrete elasticity complex \eqref{disstresscomplex}, we need
the help of the following lemma.
\begin{lemma}\label{lem:stressexact}
For $K \in \mathcal{T}_h$, it holds
\begin{equation}
\label{stressexact}
\div\mathring{\boldsymbol{\Sigma}}(K)=\mathbb{P}_1(K;\mathbb{R}^2)/\mathbf{RM}. 
\end{equation}
where $\mathring{\boldsymbol{\Sigma}}(K):=\boldsymbol{\Sigma}(K)\cap {H}_0(\div,K;\mathbb{S})$.
As a result, $\div:\mathring{\boldsymbol{\Sigma}}(K)\to\mathbb{P}_1(K;\mathbb{R}^2)/\mathbf{RM}$ is bijective.
\end{lemma}
\begin{proof}
Apparently we have $\div\mathring{\boldsymbol{\Sigma}}(K)\subseteq \mathbb{P}_1(K;\mathbb{R}^2)/\mathbf{RM}$. On the other hand, for any $\boldsymbol{p} \in \mathbb{P}_1(K;\mathbb{R}^2)/\mathbf{RM}$, there exists $\boldsymbol{\tau} \in {H}_0^1(K;\mathbb{S})$ satisfying $
\div \boldsymbol{\tau} = \boldsymbol{p}$. Let $\widetilde{\boldsymbol{\tau}} \in \mathring{\boldsymbol{\Sigma}}(K)$ be determined by $$(\widetilde{\boldsymbol{\tau}},\boldsymbol{q})_K=(\boldsymbol{\tau}, \boldsymbol{q})_K \ \ \forall~\boldsymbol{q} \in \mathbb{P}_0(K;\mathbb{S}).$$ Then
$$
(\div\widetilde{\boldsymbol{\tau}},\boldsymbol{q})_K=-(\widetilde{\boldsymbol{\tau}},\boldsymbol{\varepsilon}(\boldsymbol{q}))_K=-(\boldsymbol{\tau},\boldsymbol{\varepsilon}(\boldsymbol{q}))_K=(\div \boldsymbol{\tau}, \boldsymbol{q})_K \ \ \forall~\boldsymbol{q} \in \mathbb{P}_1(K;\mathbb{R}^2).
$$
Hence $\div \widetilde{\boldsymbol{\tau}}=\div \boldsymbol{\tau}=\boldsymbol{p}$, which ends the proof.
\end{proof}

\begin{lemma}\label{lem:elascomplex2d}
The finite element elasticity complex \eqref{disstresscomplex} is exact.
\end{lemma}
\begin{proof}
The first step is to prove that the mapping $\div: \boldsymbol{\Sigma}_h \rightarrow \boldsymbol{V}_h$ is surjective.	For any $\boldsymbol{p}_h \in \boldsymbol{V}_h$, there exists $\boldsymbol{\tau} \in {H}^1
(\Omega; \mathbb{S})$ such that $\div\boldsymbol{\tau} = \boldsymbol{p}_h.$
Let
$\boldsymbol{\tau}_h \in \boldsymbol{\Sigma}_h$ be determined by
\begin{align*}
\boldsymbol{\tau}_h(\texttt{v})&=0  \qquad\qquad\quad\; \forall~\texttt{v}\in \Delta_0(\mathcal{T}_h),\\
(\boldsymbol{n}^{\intercal}\boldsymbol{\tau}_h \boldsymbol{n},q)_e&=(\boldsymbol{n}^{\intercal}\boldsymbol{\tau} \boldsymbol{n},q)_e \quad \forall~q\in \mathbb{P}_{1}(e), e\in \partial K,\\
(\boldsymbol{t}^{\intercal}\boldsymbol{\tau}_h \boldsymbol{n},q)_e&=	(\boldsymbol{t}^{\intercal}\boldsymbol{\tau} \boldsymbol{n},q)_e \quad\; \forall~q\in \mathbb{P}_{0}(e), e\in \partial K,\\
(\boldsymbol{\tau}_h, \boldsymbol{q})_K&=(\boldsymbol{\tau}, \boldsymbol{q})_K \quad\quad\;\;  \forall~\boldsymbol{q} \in \mathbb{P}_{0}(K ; \mathbb{S}), K \in \mathcal{T}_h.
\end{align*}
Then it follows from the integration by parts that
$$
\left(\boldsymbol{p}_h-\div \boldsymbol{\tau}_h, \boldsymbol{q}\right)_K=\left(\div\left(\boldsymbol{\tau}-\boldsymbol{\tau}_h\right), \boldsymbol{q}\right)_K=0 \quad \forall~ \boldsymbol{q} \in \mathbf{RM},
$$
i.e., $\left.\left(\boldsymbol{p}_h-\div \boldsymbol{\tau}_h\right)\right|_K \in \mathbb{P}_1(K;\mathbb{R}^2)/\mathbf{RM}$ for each $K \in\mathcal{T}_h$. By (\ref{stressexact}), there exists $\widetilde{\boldsymbol{\tau}}_h \in \boldsymbol{\Sigma}_h$ such that
$$\widetilde{\boldsymbol{\tau}}_h|_K \in \mathring{\boldsymbol{\Sigma}}(K),\  \div(\widetilde{\boldsymbol{\tau}}_h|_K)=(\boldsymbol{p}_h-\div \boldsymbol{\tau}_h)|_K \quad \forall~K \in  \mathcal{T}_h.$$
Hence $\boldsymbol{\tau}_h+\widetilde{\boldsymbol{\tau}}_h \in \boldsymbol{\Sigma}_h$ satisfies $\div(\boldsymbol{\tau}_h+\widetilde{\boldsymbol{\tau}}_h)=\boldsymbol{p}_h$, that is $\div \boldsymbol{\Sigma}_h = \boldsymbol{V}_h$. 

Next we prove 
$\boldsymbol{\Sigma}_h \cap \ker(\div)=\air W_h.$
We get from the Euler’s formula $\#\Delta_{1}(\mathcal{T}_h)+1=\#\Delta_{0}(\mathcal{T}_h)+\#\mathcal T_h$ that
$$\dim \boldsymbol{\Sigma}_h \cap \ker(\div) =\dim \boldsymbol{\Sigma}_h-\dim \boldsymbol{V}_h=3 \# \Delta_{0}(\mathcal{T}_h)+3\#\Delta_{1}(\mathcal{T}_h)-3\#\mathcal{T}_h=6\#\Delta_{0}(\mathcal{T}_h)-3.$$
Therefore, $\dim\boldsymbol{\Sigma}_h \cap \ker(\div)=\dim W_h-3$, as required.
\end{proof}

The finite element elasticity complex \eqref{disstresscomplex} is demonstrated in Fig. \ref{2Dstresscomplex}
	\begin{figure}[h]
		\begin{center}
			\includegraphics[width=7cm]{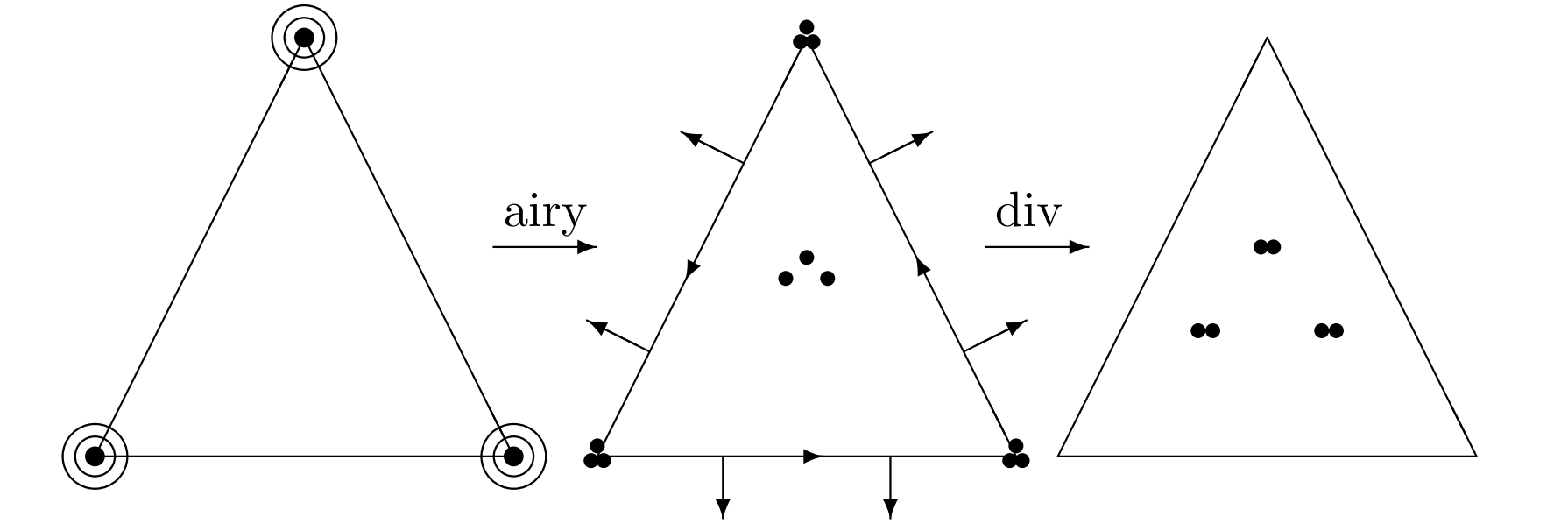}
			\caption{The finite element elasticity complex \eqref{disstresscomplex} on a triangle.  The DoFs $\nabla v(\texttt{v})$ and $\nabla^2 v(\texttt{v})$ at each vertex are shown by one circle. }
			\label{2Dstresscomplex}
		\end{center}
	\end{figure}

\subsection{A reduced finite element elasticity complex}
In the end of this section we will present a reduced version of the finite element elasticity complex \eqref{disstresscomplex} by replacing the piecewise linear polynomial space with the piecewise rigid motion space.

Take the space
	of shape functions as
	\begin{align*}
		\boldsymbol{\Sigma}^r(K)  :&=\left\{\boldsymbol{\tau} \in \mathbb{P}_3(K ; \mathbb{S}): \operatorname{div} \boldsymbol{\tau} \in \mathbf{RM},\left.\boldsymbol{t}^{\intercal} \boldsymbol{\tau} \boldsymbol{n}\right|_e \in \mathbb{P}_2(e) \quad \forall~e \in \Delta_1(K)\right\} \\
		& =\air \mathbb{P}_5^{-}(K) \oplus \operatorname{sym}(\mathbf{RM}\boldsymbol{x}^{\intercal}).
	\end{align*}
	We have $\dim \boldsymbol{\Sigma}^r(K) = 18$ and $\mathbb{P}_1(K ; \mathbb{S})\subset\boldsymbol{\Sigma}^r(K)$, but $\mathbb{P}_2(K ; \mathbb{S})\not\subseteq\boldsymbol{\Sigma}^r(K)$. The DoFs are given by
\begin{subequations}\label{dofreduced}
	\begin{align}
		\label{dofreduced1} \boldsymbol{\tau}(\texttt{v}) \quad &\forall~\texttt{v} \in \Delta_0(K),\\
		\label{dofreduced2}  \quad\left(\boldsymbol{n}^{\intercal} \boldsymbol{\tau} \boldsymbol{n}, q\right)_e \quad &\forall~q \in \mathbb{P}_1(e), e \in \partial K,\\
		\label{dofreduced3} \quad\left(\boldsymbol{t}^{\intercal} \boldsymbol{\tau} \boldsymbol{n}, q\right)_e \quad &\forall~q \in \mathbb{P}_0(e), e \in \partial K.
	\end{align}
\end{subequations}
Employing the similar argument in proving Lemma 2.1, the DoFs (\ref{dofreduced}) are unisolvent for $\boldsymbol{\Sigma}^r(K)$.

It's complicated to present explicit basis functions of $\boldsymbol{\Sigma}^r(K)$, but we can derive the basis functions of $\boldsymbol{\Sigma}^r(K)$ from the basis functions of $\boldsymbol{\Sigma}(K)$ by solving three-order linear systems. Denote the basis functions of $\boldsymbol{\Sigma}(K)$ in subsection~\ref{sub:basis} by $\boldsymbol{\phi}_i$ for $i=0,1,\ldots, 20$, where
$$
\boldsymbol{\phi}_{18}=\lambda_1\lambda_2 \boldsymbol{t}_0 \otimes \boldsymbol{t}_0, \quad \boldsymbol{\phi}_{19}=\lambda_2\lambda_0 \boldsymbol{t}_1 \otimes \boldsymbol{t}_1, \quad \boldsymbol{\phi}_{20}=\lambda_0\lambda_1 \boldsymbol{t}_2 \otimes \boldsymbol{t}_2.
$$
For $i=0,1,\ldots, 17$, take $\widetilde{\boldsymbol{\phi}}_i=\boldsymbol{\phi}_{i}+\alpha_{i, 1}\boldsymbol{\phi}_{18}+\alpha_{i, 2}\boldsymbol{\phi}_{19}+\alpha_{i, 3}\boldsymbol{\phi}_{20}$ with $\alpha_{i, 1}, \alpha_{i, 2}, \alpha_{i, 3}\in\mathbb R$ such that  satisfying $\div\widetilde{\boldsymbol{\phi}}_i$ is orthogonal to $\mathbb{P}_1(K;\mathbb{R}^2)/\mathbf{RM}$ with respect to the inner product $(\cdot, \cdot)_K$, then $\widetilde{\boldsymbol{\phi}}_i\in\boldsymbol{\Sigma}^r(K)$. 
Since $\mathbb{P}_1(K;\mathbb{R}^2)/\mathbf{RM}=\textrm{span}\{\div\boldsymbol{\phi}_{18}, \div\boldsymbol{\phi}_{19}, \div\boldsymbol{\phi}_{20}\}$ by~\eqref{stressexact},
the constant coefficients $\alpha_{i, 1}, \alpha_{i, 2}, \alpha_{i, 3}$ satisfy the linear system
$$
\begin{pmatrix}
(\div\boldsymbol{\phi}_{18}, \boldsymbol{\phi}_{18})_K & (\div\boldsymbol{\phi}_{18}, \boldsymbol{\phi}_{19})_K & (\div\boldsymbol{\phi}_{18}, \boldsymbol{\phi}_{20})_K \\
(\div\boldsymbol{\phi}_{19}, \boldsymbol{\phi}_{18})_K & (\div\boldsymbol{\phi}_{19}, \boldsymbol{\phi}_{19})_K & (\div\boldsymbol{\phi}_{19}, \boldsymbol{\phi}_{20})_K \\
(\div\boldsymbol{\phi}_{20}, \boldsymbol{\phi}_{18})_K & (\div\boldsymbol{\phi}_{20}, \boldsymbol{\phi}_{19})_K & (\div\boldsymbol{\phi}_{20}, \boldsymbol{\phi}_{20})_K
\end{pmatrix}\begin{pmatrix}
\alpha_{i, 1}\\
\alpha_{i, 2}\\
\alpha_{i, 3}
\end{pmatrix}
 =-\begin{pmatrix}
(\div\boldsymbol{\phi}_{i}, \boldsymbol{\phi}_{18})_K \\
(\div\boldsymbol{\phi}_{i}, \boldsymbol{\phi}_{19})_K \\
(\div\boldsymbol{\phi}_{i}, \boldsymbol{\phi}_{20})_K 
\end{pmatrix}.
$$
This linear system is well-posed. Solve this linear system to get $\alpha_{i, 1}, \alpha_{i, 2}, \alpha_{i, 3}$. Then we obtain $\{\widetilde{\boldsymbol{\phi}}_i\}_{i=0}^{17}$, which forms a basis of $\boldsymbol{\Sigma}^r(K)$.

Define global finite element spaces
\begin{align*}
\boldsymbol{\Sigma}_h^r&:=\{\boldsymbol{\tau}_h \in L^2(\Omega; \mathbb{S}): \boldsymbol{\tau}_h|_K\in \boldsymbol{\Sigma}^r(K)\;\; \forall\,K \in \mathcal{T}_{h}, \text{ DoFs \eqref{dofreduced} are single-valued}\}, \\
\boldsymbol{V}_h^r&:=\{\boldsymbol{q}_h \in {L}^2(\Omega ; \mathbb{R}^2):\boldsymbol{q}_h|_K \in \mathbf{RM} \quad \forall\,K \in \mathcal{T}_h\}.
\end{align*}
Clearly 
$\boldsymbol{\Sigma}_h^r\subset \boldsymbol{\Sigma}_h\subset {H}(\div,\Omega; \mathbb{S})$.

\begin{lemma}
The finite element elasticity complex
\begin{align}\label{reduceddisstresscomplex}
\mathbb{P}_1(\Omega)\xrightarrow{\subset}W_h \xrightarrow{\air}\boldsymbol{\Sigma}^r_h  \xrightarrow{\div}\boldsymbol{V}^r_h \rightarrow 0
\end{align}
is exact.
\end{lemma}
\begin{proof}
Since $\boldsymbol{\Sigma}_h^r \cap \operatorname{ker}(\div)=\boldsymbol{\Sigma}_h \cap \operatorname{ker}(\operatorname{div})$, we get from the exactness of the complex~\eqref{disstresscomplex}
that $\boldsymbol{\Sigma}_h^r \cap \operatorname{ker}(\div)=\air W_h$. 
Notice that
$$
\dim\boldsymbol{\Sigma}_h-\dim\boldsymbol{\Sigma}_h^r=\dim\boldsymbol{V}_h-\dim\boldsymbol{V}_h^r=3\# \mathcal{T}_h.
$$
Apply the exactness of the complex \eqref{disstresscomplex} to achieve
\begin{align*}	
\dim\div \boldsymbol{\Sigma}_h^r&=\dim\boldsymbol{\Sigma}_h^r-\dim(\boldsymbol{\Sigma}_h^r \cap \operatorname{ker}(\div))\\
&=\dim\boldsymbol{\Sigma}_h-\dim(\boldsymbol{\Sigma}_h \cap \operatorname{ker}(\div))-3\# \mathcal{T}_h \\
&=\dim\div \boldsymbol{\Sigma}_h-3\# \mathcal{T}_h=\dim\boldsymbol{V}_h-3\# \mathcal{T}_h=\dim\boldsymbol{V}_h^r.
\end{align*}
Hence $\operatorname{div} \boldsymbol{\Sigma}_h^r=\boldsymbol{V}_h^r$.
\end{proof}

The reduced finite element elasticity complex \eqref{reduceddisstresscomplex} is demonstrated in Fig. \ref{2Dreducedstresscomplex}.
\begin{figure}[htbp]
\begin{center}
\includegraphics[width=7cm]{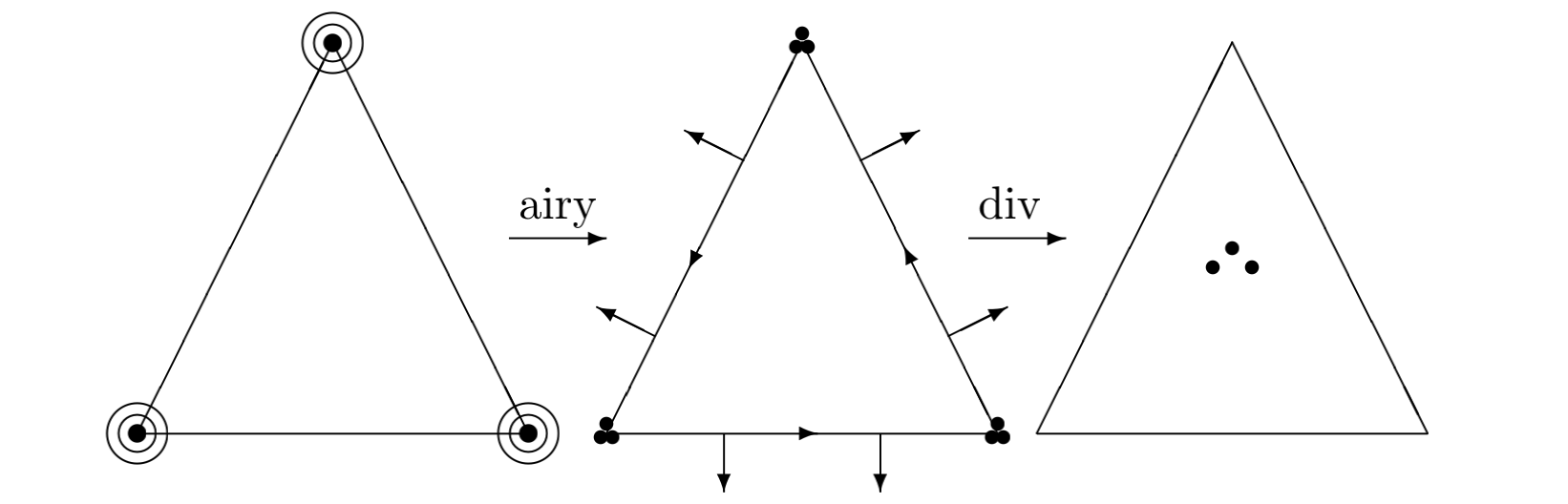}
\caption{The reduced finite element elasticity complex \eqref{reduceddisstresscomplex} on a triangle. }
\label{2Dreducedstresscomplex}
\end{center}
\end{figure}

\section{Low-order ${H}(\div; \mathbb{S})$-conforming finite elements in higher dimensions}\label{sec:divSfem}

In this section we will construct low-order ${H}(\div; \mathbb{S})$-conforming finite elements in higher dimensions, i.e. $d\geq3$.	

\subsection{Low-order finite element for symmetric tensors}

Recall the symmetric $H(\div)$-conforming finite element in \cite{ChenHuang2022}:
the DoFs
\begin{subequations}\label{HdivSBDMfemdof}
\begin{align}
\boldsymbol \tau (\texttt{v}) & \quad\forall~\texttt{v}\in \Delta_{0}(K), \label{HdivSBDMfemdof1}\\
(\boldsymbol  n_i^{\intercal}\boldsymbol \tau\boldsymbol n_j, q)_f & \quad\forall~q\in\mathbb P_{d-r}(f),  f\in\Delta_{r}(K),\;  \label{HdivSBDMfemdof2}\\
&\quad\quad 1\leq i\leq j\leq d-r, \textrm{ and } r=1,\ldots, d-2, \notag\\
(\boldsymbol  n^{\intercal}\boldsymbol \tau\boldsymbol n, q)_F & \quad\forall~q\in\mathbb P_{1}(F),  F\in\partial K,\;  \label{HdivSBDMfemdof3}\\
(\Pi_F\boldsymbol \tau\boldsymbol n, \boldsymbol q)_F & \quad\forall~\boldsymbol q\in {\rm ND}_{d-1}(F),  F\in\partial K,\label{HdivSBDMfemdof4} \\
 (\div\boldsymbol \tau, \boldsymbol q)_K &\quad~\forall~\boldsymbol q\in   \mathbb P_{d}(K;\mathbb R^d)/\mathbf{RM}, \label{HdivSBDMfemdof5}\\
 (\boldsymbol \tau, \boldsymbol q)_K &\quad~\forall~\boldsymbol q\in \ker (\cdot\boldsymbol x)\cap \mathbb P_{d-1}(K;\mathbb S) \label{HdivSBDMfemdof6}
\end{align}
\end{subequations}
are uni-solvent for $\mathbb P_{d+1}(K;\mathbb S)$.
Define the normal-normal face bubble space
$$
\mathbb B_{\partial K}^{nn}:=\{\boldsymbol{\tau}\in\mathbb P_{d+1}(K;\mathbb S): \textrm{all the DoFs in \eqref{HdivSBDMfemdof} except \eqref{HdivSBDMfemdof3} vanish}\}.
$$
Clearly, $\dim\mathbb B_{\partial K}^{nn}=d(d+1)$.
By the vanishing DoF \eqref{HdivSBDMfemdof5}, we have $\div \mathbb B_{\partial K}^{nn}\subseteq\mathbf{RM}$. For $\boldsymbol{\tau}\in\mathbb B_{\partial K}^{nn}$, $\Pi_F(\boldsymbol \tau\boldsymbol n)=0$ for $F\in\partial K$, and $(\boldsymbol  n_i^{\intercal}\boldsymbol \tau\boldsymbol n_j)|_f=0$ for $f\in\Delta_{r}(K)$,
$1\leq i\leq j\leq d-r$, $r=0,\ldots, d-2$.

\begin{lemma}
For $d\geq 3$, $\mathbb P_{2}(K;\mathbb S)\cap\mathbb B_{\partial K}^{nn}=\{0\}$.
\end{lemma}
\begin{proof}
For $\boldsymbol{\tau}\in\mathbb P_{2}(K;\mathbb S)\cap\mathbb B_{\partial K}^{nn}$, 
by $(\boldsymbol{n}^{\intercal}\boldsymbol{\tau}\boldsymbol{n})|_{F}\in\mathbb P_2(F)$,
the vanishing DoFs \eqref{HdivSBDMfemdof1}-\eqref{HdivSBDMfemdof2} implies $(\boldsymbol{n}^{\intercal}\boldsymbol{\tau}\boldsymbol{n})|_{F}=0$, then all the DoFs \eqref{HdivSBDMfemdof} vanish. Thus $\boldsymbol{\tau}=0$.
\end{proof}

\begin{remark}\rm
For $F\in\partial K$, let $\mathbb B_{F}^{nn}:=\{\boldsymbol{\tau}\in\mathbb B_{\partial K}^{nn}:(\boldsymbol{n}^{\intercal}\boldsymbol{\tau}\boldsymbol{n})|_{\partial K\backslash F}=0\}$. Clearly we have $\mathbb B_{\partial K}^{nn}=\Oplus_{F\in\partial K}\mathbb B_{F}^{nn}$, and
\begin{equation*}
\mathbb B_{F}^{nn}=\big\{\boldsymbol{\tau}\in (b_F\mathbb P_1(F)\boldsymbol{n}\otimes\boldsymbol{n})\, \oplus \,\mathbb B_{d+1}(\div, K;\mathbb S): \textrm{all the DoFs \eqref{HdivSBDMfemdof5}-\eqref{HdivSBDMfemdof6} vanish} \big\},
\end{equation*}
where the bubble space \cite[Lemma 2.2]{Hu2015Higher}
$$
\mathbb B_{d+1}(\div, K;\mathbb S):=\mathbb P_{d+1}(K;\mathbb S)\cap H_0(\div, K;\mathbb{S})=\Oplus_{0\leq i<j\leq d}\lambda_i\lambda_j\mathbb P_{d-1}(K)\boldsymbol{t}_{ij}\otimes\boldsymbol{t}_{ij}.
$$
A basis of $\mathbb B_{F}^{nn}$ can be derived by solving some local linear systems with the same coefficient matrix of order $\dim\mathbb P_{d-1}(K)={2d-1\choose d}$.
\end{remark}

By enriching $\mathbb P_{2}(K;\mathbb S)$ with the normal-normal face bubble space,
take the space of shape functions as
$$
\boldsymbol{\Sigma}(K):=\mathbb P_{2}(K;\mathbb S)\oplus\mathbb B_{\partial K}^{nn}.
$$
The dimension of $\boldsymbol{\Sigma}(K)$ is $\frac{1}{4}d(d+1)^2(d+2)+d(d+1)=\frac{1}{4}d(d+1)(d^2+3d+6)$.
The DoFs for $\boldsymbol{\Sigma}(K)$ are given by
\begin{subequations}\label{lowdivShddof}
\begin{align}
\boldsymbol \tau (\texttt{v}) & \quad\forall~\texttt{v}\in \Delta_{0}(K), \label{lowdivShddof1}\\
(\boldsymbol  n_i^{\intercal}\boldsymbol \tau\boldsymbol n_j, q)_e & \quad\forall~q\in\mathbb P_{0}(e),  e\in\Delta_{1}(K), 1\leq i\leq j\leq d-1,\;  \label{lowdivShddof2}\\
(\boldsymbol  n^{\intercal}\boldsymbol \tau\boldsymbol n, q)_F & \quad\forall~q\in\mathbb P_{1}(F),  F\in\partial K,\;  \label{lowdivShddof3}\\
(\Pi_F\boldsymbol \tau\boldsymbol n, \boldsymbol q)_F & \quad\forall~\boldsymbol q\in {\rm ND}_{0}(F),  F\in\partial K,\label{lowdivShddof4} \\
 (\boldsymbol \tau, \boldsymbol q)_K &\;\;\;\,\forall~\boldsymbol q\in \mathbb P_{0}(K;\mathbb S) \label{lowdivShddof5}.
\end{align}
\end{subequations}

\begin{lemma}\label{lem:unisoldgeq3}
Let $\boldsymbol{\tau} \in \boldsymbol{\Sigma}(K)$ and $F\in\partial K$. Assume all the DoFs \eqref{lowdivShddof1}-\eqref{lowdivShddof4} on $F$ vanish, then $(\boldsymbol{\tau}\boldsymbol{n})|_F=0$.
\end{lemma}
\begin{proof}
Let $\boldsymbol{\tau}=\boldsymbol{\tau}_1+\boldsymbol{\tau}_2$ with $\boldsymbol{\tau}_1\in\mathbb P_{2}(K;\mathbb S)$ and $\boldsymbol{\tau}_2\in\mathbb B_{\partial K}^{nn}$. By the vanishing DoFs \eqref{lowdivShddof1}-\eqref{lowdivShddof2} and \eqref{lowdivShddof4}, we have $\Pi_F(\boldsymbol \tau_1\boldsymbol n)=0$, and $(\boldsymbol n_i^{\intercal}\boldsymbol \tau_1\boldsymbol n_j)|_e=0$ for $e\in\Delta_{1}(F)$,
$1\leq i\leq j\leq d-1$. Hence $(\boldsymbol{\tau}_1\boldsymbol{n})|_F=0$. Now $(\boldsymbol{\tau}\boldsymbol{n})|_F=(\boldsymbol{\tau}_2\boldsymbol{n})|_F$, so $(\boldsymbol{\tau}\boldsymbol{n})|_F=0$ follows from the vanishing DoF \eqref{lowdivShddof3}.
\end{proof}

\begin{lemma}\label{lem:unisoldgeq32}
The DoFs \eqref{lowdivShddof} are unisolvent for $\boldsymbol{\Sigma}(K)$.
\end{lemma}
\begin{proof}
The number of DoFs \eqref{lowdivShddof} is
$$
\frac{1}{2}d(d+1)^2+\frac{1}{4}d^2(d+1)(d-1) +d(d+1) +\frac{1}{2}(d+1)d(d-1)+\frac{1}{2}d(d+1)=\frac{1}{4}d(d+1)(d^2+3d+6),
$$
which is same as
the dimension of $\boldsymbol{\Sigma}(K)$. 

Take $\boldsymbol{\tau}=\boldsymbol{\tau}_1+\boldsymbol{\tau}_2 \in \boldsymbol{\Sigma}(K)$ with $\boldsymbol{\tau}_1\in\mathbb P_{2}(K;\mathbb S)$ and $\boldsymbol{\tau}_2\in\mathbb B_{\partial K}^{nn}$, and suppose all the DoFs \eqref{lowdivShddof} vanish.
By Lemma~\ref{lem:unisoldgeq3}, $\boldsymbol{\tau}\boldsymbol{n}|_{\partial K}=0$ and $\boldsymbol{\tau}_2\boldsymbol{n}|_{\partial K}=0$. Then $\boldsymbol{\tau}_2=0$. Hence, $\boldsymbol{\tau}\in\mathbb P_{2}(K;\mathbb S)$ satisfies $\boldsymbol{\tau}\boldsymbol{n}|_{\partial K}=0$.
Finally, $\boldsymbol{\tau}=0$ holds from the vanishing DoF \eqref{lowdivShddof5}.
\end{proof}

As Lemma~\ref{lem:stressexact}, $\div:\boldsymbol{\Sigma}(K)\cap {H}_0(\div,K;\mathbb{S})\to\mathbb{P}_1(K;\mathbb{R}^d)/\mathbf{RM}$ is bijective, for each $K\in\mathcal T_h$.

Define global finite element spaces
\begin{align*}
\boldsymbol{\Sigma}_h&:=\{\boldsymbol{\tau}_h \in L^2(\Omega; \mathbb{S}): \boldsymbol{\tau}_h|_K\in \boldsymbol{\Sigma}(K)\;\; \forall\,K \in \mathcal{T}_{h}, \text{ DoFs \eqref{lowdivShddof} are single-valued}\}, \\
\boldsymbol{V}_h&:=\{\boldsymbol{q}_h\in {L}^{2}(\Omega;\mathbb{R}^{d}): \boldsymbol{q}_h|_K\in\mathbb{P}_{1}(K ; \mathbb{R}^{d})\quad \forall~K \in \mathcal{T}_{h}\}.
\end{align*}
By Lemma~\ref{lem:unisoldgeq3}, $\boldsymbol{\Sigma}_h\subset {H}(\div,\Omega; \mathbb{S})$.
Apply the same argument as in Lemma~\ref{lem:elascomplex2d} to acquire
\begin{equation}\label{eq:divonto}
\div\boldsymbol{\Sigma}_h=\boldsymbol{V}_h.
\end{equation}

The basis functions of $\boldsymbol{\Sigma}(K)$ for $d\geq3$ are given as follows:
\begin{itemize}
\item[(i)] For vertex $\texttt{v}_i$ with $i=0, 1,\ldots, d$, the basis functions related to DoF (\ref{lowdivShddof1}) are $\lambda_i^2 \mathbb S$.
\item[(ii)] For edge $e$, the basis functions related to DoF~(\ref{lowdivShddof2}) are $b_e \sym(\boldsymbol{n}_i \otimes \boldsymbol{n}_j)$ for $1\leq i\leq j\leq d-1$.
\item[(iii)] For face $F$, the basis functions related to DoF~(\ref{lowdivShddof3}) are
those of $\mathbb B_{F}^{nn}$.
\item[(iv)] For face $F$, the basis functions related to DoF~(\ref{lowdivShddof4}) are $b_e \sym(\boldsymbol{t}_e \otimes \boldsymbol{n}_F)$ for $e\in \Delta_1(F)$.
\item[(v)] The basis functions related to DoF (\ref{lowdivShddof5}) are $b_e \boldsymbol{t}_e \otimes \boldsymbol{t}_e$ for $e\in \Delta_1(K)$.
\end{itemize}

\subsection{A reduced finite element for symmetric tensors}
Now we reduce $\boldsymbol{\Sigma}(K)$ to the shape function space
\begin{align*}
\boldsymbol{\Sigma}^r(K)  :=\{\boldsymbol{\tau} \in \boldsymbol{\Sigma}(K):& \operatorname{div} \boldsymbol{\tau} \in \mathbf{RM}, \\
&(\boldsymbol n_i^{\intercal}\boldsymbol \tau\boldsymbol n_j)|_e\in\mathbb P_1(e) \textrm{ for } e\in\Delta_{1}(K), 1\leq i\leq j\leq d-1\}.
\end{align*}
It is easy to see that $\mathbb P_1(K;\mathbb S)\subset\boldsymbol{\Sigma}^r(K)$.

The DoFs for $\boldsymbol{\Sigma}^r(K)$ are given by
\begin{subequations}\label{reducedivShddof}
\begin{align}
\boldsymbol \tau (\texttt{v}) & \quad\forall~\texttt{v}\in \Delta_{0}(K), \label{reducedivShddof1}\\
(\boldsymbol  n^{\intercal}\boldsymbol \tau\boldsymbol n, q)_F & \quad\forall~q\in\mathbb P_{1}(F),  F\in\partial K,\;  \label{reducedivShddof3}\\
(\Pi_F\boldsymbol \tau\boldsymbol n, \boldsymbol q)_F & \quad\forall~\boldsymbol q\in {\rm ND}_{0}(F),  F\in\partial K. \label{reducedivShddof4}
\end{align}
\end{subequations}
\begin{lemma}
The DoFs \eqref{reducedivShddof} are unisolvent for $\boldsymbol{\Sigma}^r(K)$.
\end{lemma}
\begin{proof}
The number of DoFs \eqref{reducedivShddof} is
$$
\frac{1}{2}d(d+1)^2 +d(d+1) +\frac{1}{2}(d+1)d(d-1)=d(d+1)^2\leq \dim\boldsymbol{\Sigma}^r(K).
$$

Take $\boldsymbol{\tau} \in \boldsymbol{\Sigma}^r(K)$, and suppose all the DoFs \eqref{reducedivShddof} vanish. The vanishing DoF \eqref{reducedivShddof1} indicates $(\boldsymbol n_i^{\intercal}\boldsymbol \tau\boldsymbol n_j)|_e=0$ for $e\in\Delta_{1}(K)$, $1\leq i\leq j\leq d-1$. Then apply Lemma~\ref{lem:unisoldgeq3} to get $\boldsymbol{\tau}\boldsymbol{n}|_{\partial K}=0$. By $\div \boldsymbol{\tau} \in \mathbf{RM}$, we get from the integration by parts that $\|\div \boldsymbol{\tau}\|_{0,K}^2=0$. Hence, $\div \boldsymbol{\tau}=0$, which implies $(\boldsymbol{\tau}, \boldsymbol{q})_{K}=0$ for $\boldsymbol{q}\in\mathbb P_{0}(K;\mathbb S)$. Finally, employ Lemma~\ref{lem:unisoldgeq32} to conclude $\boldsymbol{\tau}=0$.
\end{proof}

The basis functions of $\boldsymbol{\Sigma}^r(K)$ for $d\geq3$ are given as follows:
\begin{itemize}
\item[(i)] For vertex $\texttt{v}_i$ with $i=0, 1,\ldots, d$, the basis functions related to DoF (\ref{reducedivShddof1}) are $\lambda_i \mathbb S$.
\item[(ii)] For face $F$, the basis functions related to DoF~(\ref{reducedivShddof3}) are
those of $\mathbb B_{F}^{nn}$.
\item[(iii)] For face $F$, the basis functions related to DoF~(\ref{reducedivShddof4}) 
form the space
\begin{align*}
\mathbb B_{F}^{tn}=\big\{\boldsymbol{\tau}\in \mathbb B_{d+1}(\div, K;\mathbb S)\oplus&\Oplus_{e\in \Delta_1(F)}b_e \sym(\boldsymbol{t}_e \otimes \boldsymbol{n}_F): \\
&\quad\quad\quad\textrm{DoFs \eqref{HdivSBDMfemdof5}-\eqref{HdivSBDMfemdof6} vanish} \big\}.
\end{align*}
Again, a basis of $\mathbb B_{F}^{tn}$ can be derived by solving some local linear systems with the same coefficient matrix of order $\dim\mathbb P_{d-1}(K)={2d-1\choose d}$.
\end{itemize}

\begin{remark}\rm
The local dimension of the first-order $H(\div)$-conforming finite element for symmetric tensors in \cite{HuZhang2016lower} is also $d(d+1)^2$. Their finite element space is defined by enriching the symmetric tensor-valued linear element space with both the $(d+1)$-order normal-normal face bubble space and the $(d+1)$-order tangential-normal face bubble space, while our finite element space is enriched by the $(d+1)$-order normal-normal face bubble space and the second order tangential-normal face bubble space. Thereby, our reduced finite element for symmetric tensors is simpler than the first order one in \cite{HuZhang2016lower} in the sense that on each face $F\in\partial K$, $(\Pi_F\boldsymbol \tau\boldsymbol n)|_F\in\mathbb P_2(F;\mathbb R^{d-1})$ for $\boldsymbol{\tau}\in\boldsymbol{\Sigma}^r(K)$, while $(\Pi_F\boldsymbol \tau\boldsymbol n)|_F\in\mathbb P_{d+1}(F;\mathbb R^{d-1})$ for the first order symmetric tensor $\boldsymbol \tau$ in \cite{HuZhang2016lower}.
Especially, we present the explicit expressions of the basis functions of $\boldsymbol{\Sigma}(K)$ and $\boldsymbol{\Sigma}^r(K)$ in terms of barycentric coordinates.
\end{remark}

Define global finite element spaces
\begin{align*}
\boldsymbol{\Sigma}_h^r&:=\{\boldsymbol{\tau}_h \in L^2(\Omega; \mathbb{S}): \boldsymbol{\tau}_h|_K\in \boldsymbol{\Sigma}^r(K)\;\; \forall\,K \in \mathcal{T}_{h}, \text{ DoFs \eqref{reducedivShddof} are single-valued}\}, \\
\boldsymbol{V}_h^r&:=\{\boldsymbol{q}_h \in {L}^2(\Omega ; \mathbb{R}^d):\boldsymbol{q}_h|_K \in \mathbf{RM} \quad \forall\,K \in \mathcal{T}_h\}.
\end{align*}
Then 
$\boldsymbol{\Sigma}_h^r\subset \boldsymbol{\Sigma}_h\subset {H}(\div,\Omega; \mathbb{S})$.

\begin{lemma}\label{lem:PiFRM}
For $F\in\partial K$, we have $\Pi_F\mathbb K\Pi_F=\mathbb K_F$ and $\Pi_F\mathbf{RM}={\rm ND}_0(F)$, where $\mathbb K_F$ is the space of skew-symmetric matrices on $F$.
\end{lemma}
\begin{proof}
Since
$$
\mathbb K={\rm span}\{\skw(\boldsymbol{t}_i\otimes\boldsymbol{t}_j), 1\leq i<j\leq d-1\}\oplus {\rm span}\{\skw(\boldsymbol{t}_i\otimes\boldsymbol{n}), 1\leq i\leq d-1\},
$$
it follows
$$
\Pi_F\mathbb K\Pi_F={\rm span}\{\skw(\boldsymbol{t}_i\otimes\boldsymbol{t}_j), 1\leq i<j\leq d-1\}=\mathbb K_F.
$$
By 
$\mathbf{RM}=\mathbb P_{0}(K;\mathbb R^d) \oplus \mathbb K\boldsymbol x$ and $(\boldsymbol x\cdot\boldsymbol n)|_F\in\mathbb P_{0}(F)$,
$$
\Pi_F\mathbf{RM}=\mathbb P_{0}(F;\mathbb R^{d-1}) + \Pi_F\mathbb K\boldsymbol x=\mathbb P_{0}(F;\mathbb R^{d-1}) + \Pi_F\mathbb K\Pi_F(\Pi_F\boldsymbol x).
$$
Therefore $\Pi_F\mathbf{RM}={\rm ND}_0(F)$ holds from $\Pi_F\mathbb K\Pi_F=\mathbb K_F$.
\end{proof}

\begin{lemma}
We have
$$
\div\boldsymbol{\Sigma}_h^r=\boldsymbol{V}_h^r.
$$	
\end{lemma}
\begin{proof}
Since $\div\boldsymbol{\Sigma}_h^r\subseteq\boldsymbol{V}_h^r$, it suffices to prove the other side.
For any $\boldsymbol{p}_h \in \boldsymbol{V}_h^r$, there exists $\boldsymbol{\tau} \in {H}^1
(\Omega; \mathbb{S})$ such that $\div\boldsymbol{\tau} = \boldsymbol{p}_h.$
Let
$\boldsymbol{\tau}_h \in \boldsymbol{\Sigma}_h^r$ be determined by
\begin{align*}
\boldsymbol{\tau}_h(\texttt{v})&=0  \qquad\qquad\quad\; \forall~\texttt{v}\in \Delta_0(\mathcal{T}_h),\\
(\boldsymbol{n}^{\intercal}\boldsymbol{\tau}_h \boldsymbol{n},q)_F&=(\boldsymbol{n}^{\intercal}\boldsymbol{\tau} \boldsymbol{n},q)_F \quad \forall~q\in \mathbb{P}_{1}(F), F\in \Delta_{d-1}(\mathcal{T}_h),\\
(\Pi_F\boldsymbol{\tau}_h \boldsymbol{n}, \boldsymbol{q})_F&=	(\Pi_F\boldsymbol{\tau} \boldsymbol{n}, \boldsymbol{q})_F \;\;\, \forall~\boldsymbol{q}\in {\rm ND}_0(F), F\in \Delta_{d-1}(\mathcal{T}_h).
\end{align*}
By Lemma~\ref{lem:PiFRM}, $\Pi_F\mathbf{RM}={\rm ND}_0(F)$, hence
$$
\left(\boldsymbol{p}_h-\div \boldsymbol{\tau}_h, \boldsymbol{q}\right)_K=\left(\div\left(\boldsymbol{\tau}-\boldsymbol{\tau}_h\right), \boldsymbol{q}\right)_K=0 \quad \forall~ \boldsymbol{q} \in \mathbf{RM}.
$$
Therefore $\div \boldsymbol{\tau}_h=\boldsymbol{p}_h$.
\end{proof}

\section{A mixed finite element method for linear elasticity}\label{sec:mfemelas}

In this section, we will apply the previously constructed ${H}(\div; \mathbb{S})$-conforming finite elements in Section~\ref{sec:femelascomplex} and Section~\ref{sec:divSfem} to advance a robust mixed finite element method for the linear elasticity problem.

\subsection{Linear elasticity problem and mixed formulation}
Consider the linear elasticity
problem under the load $\boldsymbol{f} \in {L}^{2}(\Omega; \mathbb{R}^{d})$ with $d\geq2$
	\begin{equation}\label{linearelasticity}
		\left\{\begin{aligned}
			\mathcal{A} \boldsymbol{\sigma}-\boldsymbol{\varepsilon}(\boldsymbol{u})&=0 \qquad \text { in } \Omega, \\
			\operatorname{div} \boldsymbol{\sigma}&=-\boldsymbol{f} \quad \text { in } \Omega, \\
			\boldsymbol{u}&=0  \qquad\text { on } \partial \Omega,
		\end{aligned}\right.
	\end{equation}
where $\boldsymbol{u}: \Omega \mapsto \mathbb{R}^d$ and $\boldsymbol{\sigma} : \Omega \mapsto \mathbb{S}$ are the displacement and stress respectively, and
the compliance tensor $\mathcal{A} : \mathbb{S} \mapsto \mathbb{S}$ is defined by
	$$
	\mathcal{A} \boldsymbol{\sigma}:=\frac{1}{2 \mu}\left(\boldsymbol{\sigma} -\frac{\lambda}{d\lambda+2 \mu}(\operatorname{tr} \boldsymbol{\sigma}) \boldsymbol{I}\right).
	$$
	Here $\boldsymbol{I}$ is the identity tensor, $\tr$ is the trace operator, and the positive constants $\lambda$ and $\mu$ are the Lam\'{e} coefficients.

The mixed formulation of the linear elasticity problem \eqref{linearelasticity} based on the  Hellinger-Reissner principle is to find $(\boldsymbol{\sigma}, \boldsymbol{u}) \in  {H}(\div, \Omega; \mathbb{S}) \times  {L}^{2}(\Omega; \mathbb{R}^{d})$ such that
	\begin{align}
		\label{HJ1}	a\left(\boldsymbol{\sigma}, \boldsymbol{\tau}\right)+b\left(\boldsymbol{\tau},\boldsymbol{u}\right) &=0  \qquad\qquad\;\;\; \forall~\boldsymbol{\tau} \in  {H}(\div, \Omega; \mathbb{S}), \\
		\label{HJ2}	b\left(\boldsymbol{\sigma},\boldsymbol{v} \right) &=-(\boldsymbol{f}, \boldsymbol{v})  \qquad \forall~\boldsymbol{v} \in {L}^{2}(\Omega; \mathbb{R}^{d}),
	\end{align}
	where the bilinear forms
	\begin{align*}
		a\left(\boldsymbol{\sigma}, \boldsymbol{\tau}\right):=(\mathcal{A}\boldsymbol{\sigma}, \boldsymbol{\tau}),\quad 
		b\left(\boldsymbol{\tau},\boldsymbol{v} \right):=(\div\boldsymbol{\tau},  \boldsymbol{v}).
	\end{align*}
	
	Clearly we have the boundedness
	\begin{align*}
		a\left(\boldsymbol{\sigma}, \boldsymbol{\tau}\right)&\leq \|\boldsymbol{\sigma}\|_{{H}(\div)}\|\boldsymbol{\tau}\|_{{H}(\div)} \quad \forall~\boldsymbol{\sigma}, \boldsymbol{\tau} \in  {H}(\div, \Omega; \mathbb{S}),\\
		b\left(\boldsymbol{\tau},\boldsymbol{v} \right)&\leq \|\boldsymbol{\tau}\|_{{H}(\div)}\|\boldsymbol{v}\|_0 \qquad\quad\, \forall~\boldsymbol{\tau} \in  {H}(\div, \Omega; \mathbb{S}), \boldsymbol{v} \in {L}^{2}(\Omega; \mathbb{R}^{d}).
	\end{align*}
The mixed formulation \eqref{HJ1}-\eqref{HJ2} is well-posed
\cite{BoffiBrezziFortin2013}.

	
\subsection{Mixed finite element method}\label{sec32}
	Based on the mixed formulation (\ref{HJ1})-(\ref{HJ2}), we propose a mixed finite element method for the linear elasticity problem (\ref{linearelasticity})
	as follows: find
	$(\boldsymbol{\sigma}_h, \boldsymbol{u}_h) \in \boldsymbol{\Sigma}_{h} \times \boldsymbol{V}_h $
	such that
	\begin{align}
		\label{distype1}	a(\boldsymbol{\sigma}_{h}, \boldsymbol{\tau}_{h})+b(\boldsymbol{\tau}_{h},\boldsymbol{u}_h) &=0 \qquad\qquad\; \forall~\boldsymbol{\tau}_{h} \in \boldsymbol{\Sigma}_{h}, \\
		\label{distype2}	b(\boldsymbol{\sigma}_{h},
		\boldsymbol{v}_h ) &=-(\boldsymbol{f}, \boldsymbol{v}_h) \quad \forall~\boldsymbol{v}_{h} \in \boldsymbol{V}_h.
	\end{align}
Thanks to $\div\boldsymbol{\Sigma}_{h}=\boldsymbol{V}_h$ from the discrete elasticity complex \eqref{disstresscomplex} and \eqref{eq:divonto}, equation~\eqref{distype2} is equivalent to
\begin{equation}\label{eq:divsigmah}
\div\boldsymbol{\sigma}_{h}=-Q_h\boldsymbol{f},
\end{equation}
where $Q_h$ is the $L^2$ orthogonal projector from ${L}^2(\Omega; \mathbb{R}^d)$ onto $\boldsymbol{V}_h$.
Applying $\div\boldsymbol{\Sigma}_{h}=\boldsymbol{V}_h$ again, we have
\begin{equation}\label{buortho}
b(\boldsymbol{\tau}_h, \boldsymbol{v}-Q_h\boldsymbol{v})=0 \quad \forall~\boldsymbol{\tau}_h\in\boldsymbol{\Sigma}_h, \boldsymbol{v}\in{L}^2(\Omega;\mathbb{R}^d).
\end{equation}

To show the well-posedness of the mixed finite element method \eqref{distype1}-\eqref{distype2}, we introduce some interpolation operators. Using the average technique, define an $H^1$-bounded interpolation operator $I_h: {H}^{1}(\Omega; \mathbb{S})\to\boldsymbol{\Sigma}_{h}$ by
\begin{align*}
(I_h\boldsymbol{\tau})(\texttt{v})&=\frac{1}{\#\mathcal{T}_{\texttt{v}}} \sum_{K\in\mathcal{T}_{\texttt{v}}}(Q_K^2\boldsymbol{\tau})(\texttt{v}) \qquad\qquad\quad\forall~\texttt{v}\in\Delta_0(\mathcal{T}_{h}),\\
(\boldsymbol{n}_i^{\intercal}(I_h\boldsymbol{\tau})\boldsymbol{n}_j, q)_e &= \frac{1}{\#\mathcal{T}_{e}} \sum_{K\in\mathcal{T}_{e}}(\boldsymbol{n}_i^{\intercal}(Q_K^2\boldsymbol{\tau})\boldsymbol{n}_j, q)_e \quad\;\;\forall~q\in \mathbb{P}_0(e), e\in\Delta_1(\mathcal{T}_{h}), \\
&\qquad\qquad\qquad\qquad\qquad\qquad\qquad\quad\;\;\, 1\leq i\leq j\leq d-1, \textrm{ if } d\geq3,\\
(\boldsymbol{n}^{\intercal}(I_h\boldsymbol{\tau})\boldsymbol{n}, q)_F &= (\boldsymbol{n}^{\intercal}\boldsymbol{\tau}\boldsymbol{n}, q)_F \qquad\qquad\qquad\quad\quad\;\;\forall~q\in \mathbb{P}_1(F), F\in\Delta_{d-1}(\mathcal{T}_{h}),\\
(\Pi_F(I_h\boldsymbol{\tau})\boldsymbol{n}, \boldsymbol{q})_F &= (\Pi_F\boldsymbol{\tau}\boldsymbol{n}, \boldsymbol{q})_F \qquad\qquad\quad\qquad\quad\,\forall~\boldsymbol q\in {\rm ND}_{0}(F), F\in\Delta_{d-1}(\mathcal{T}_{h}),\\
(I_h\boldsymbol{\tau}, \boldsymbol{q})_K &= (\boldsymbol{\tau}, \boldsymbol{q})_K  \qquad\qquad\qquad\qquad\qquad\,\forall~\boldsymbol{q}\in \mathbb{P}_0(K;\mathbb{S}), K\in\mathcal{T}_{h},
\end{align*}
where $\mathcal{T}_{\delta}$ and $\mathcal{T}_{e}$ are the sets of the simplices in $\mathcal{T}_h$ sharing the common vertex $\delta$ and edge $e$ respectively. 
Following the argument in \cite[Section~2.3]{HuangHuang2011}, for $1\leq s\leq 3$ we have the estimate
\begin{equation}\label{eq:Iherrorestimate}
\sum_{i=0}^s\sum_{K\in\mathcal{T}_h}h_K^{2i-2s}|\boldsymbol{\tau}-I_h\boldsymbol{\tau}|_{i,K}^2\lesssim |\boldsymbol{\tau}|_s^2\quad\forall~\boldsymbol{\tau}\in H^s(\Omega;\mathbb S).
\end{equation}
The interpolation operator $I_h$ is $H^1$-bounded, and has the optimal error estimate. Apply the integration by parts to acquire
\begin{equation}\label{IhcommuRM}
(\div(\boldsymbol{\tau}-I_h\boldsymbol{\tau}), \boldsymbol{v})_K=0\quad\forall~\boldsymbol{v}\in\mathbf{RM}, K\in\mathcal{T}_h.
\end{equation}

We also need an interpolation operator commuting with the $\div$ operator.
Define interpolation operator $I_h^b: {H}(\div, \Omega; \mathbb{S})\to\boldsymbol{\Sigma}_{h}$ as follows: for $\boldsymbol{\tau}\in{H}(\div, \Omega; \mathbb{S})$, let $I_h^b\boldsymbol{\tau}\in\boldsymbol{\Sigma}_{h}$ be determined by $(I_h^b\boldsymbol{\tau})|_K\in{H}_0(\div, K; \mathbb{S})$ for $K\in\mathcal{T}_h$, and
$$
(I_h^b\boldsymbol{\tau}, \boldsymbol{\varepsilon}(\boldsymbol{v}))_K=-(\div\boldsymbol{\tau}, \boldsymbol{v})_K\quad\forall~\boldsymbol{v}\in\mathbb{P}_1(K;\mathbb{R}^d)/\mathbf{RM}.
$$
Notice that $\boldsymbol{\varepsilon}(\mathbb{P}_1(K;\mathbb{R}^d)/\mathbf{RM})=\mathbb{P}_0(K;\mathbb{S})$.

\begin{lemma}
For $\boldsymbol{\tau}\in{H}(\div, \Omega; \mathbb{S})$ and $K\in\mathcal{T}_h$, 
we have
\begin{equation}\label{eq:Ihbboundedness}
\|I_h^b\boldsymbol{\tau}\|_{0,K}\lesssim h_K\|\div\boldsymbol{\tau}\|_{0,K},	
\end{equation}
\begin{equation}\label{eq:Ihbcommutatediv}
(\div(\boldsymbol{\tau}-I_h^b\boldsymbol{\tau}), \boldsymbol{v})_K=0\quad\forall~\boldsymbol{v}\in\mathbb{P}_1(K;\mathbb{R}^d)/\mathbf{RM}.
\end{equation}
\end{lemma}
\begin{proof}
Notice that $\boldsymbol{\varepsilon}(\mathbb{P}_1(K;\mathbb{R}^d)/\mathbf{RM})=\boldsymbol{\varepsilon}({\mathbb{P}}_1(K;\mathbb{R}^d))=\mathbb{P}_0(K;\mathbb{S})$.
By the scaling argument and the second Korn's inequality \cite[Remark 1.1]{Brenner2004},
\begin{align*}
\|I_h^b\boldsymbol{\tau}\|_{0,K}&\eqsim \|Q_K(I_h^b\boldsymbol{\tau})\|_{0,K}=\sup_{\boldsymbol{v}\in\mathbb{P}_1(K;\mathbb{R}^d)/\boldsymbol{RM}}\frac{(I_h^b\boldsymbol{\tau}, \boldsymbol{\varepsilon}(\boldsymbol{v}))_K}{\|\boldsymbol{\varepsilon}(\boldsymbol{v})\|_{0,K}}=\sup_{\boldsymbol{v}\in\mathbb{P}_1(K;\mathbb{R}^d)/\boldsymbol{RM}}\frac{-(\div\boldsymbol{\tau}, \boldsymbol{v})_K}{\|\boldsymbol{\varepsilon}(\boldsymbol{v})\|_{0,K}} \\
&\leq \sup_{\boldsymbol{v}\in\mathbb{P}_1(K;\mathbb{R}^d)/\boldsymbol{RM}}\frac{\|\div\boldsymbol{\tau}\|_{0,K}\|\boldsymbol{v}\|_{0,K}}{\|\boldsymbol{\varepsilon}(\boldsymbol{v})\|_{0,K}} \lesssim h_K\|\div\boldsymbol{\tau}\|_{0,K}.
\end{align*}
Hence the boundedness \eqref{eq:Ihbboundedness} holds.

It follows from the integration by parts and the definition of $I_h^b$ that
$$
(\div(\boldsymbol{\tau}-I_h^b\boldsymbol{\tau}), \boldsymbol{v})_K=(\div\boldsymbol{\tau}, \boldsymbol{v})_K + (I_h^b\boldsymbol{\tau}, \boldsymbol{\varepsilon}(\boldsymbol{v}))_K=0
$$
for $\boldsymbol{v}\in\mathbb{P}_1(K;\mathbb{R}^d)/\mathbf{RM}$, which is exactly \eqref{eq:Ihbcommutatediv}.
\end{proof}

Combining $I_h$ and $I_h^b$, we define interpolation operator $\Pi_h: {H}^{1}(\Omega; \mathbb{S})\to\boldsymbol{\Sigma}_{h}$ as $\Pi_h\boldsymbol{\tau}:=I_h\boldsymbol{\tau}+I_h^b(\boldsymbol{\tau}-I_h\boldsymbol{\tau})$.

\begin{lemma}
We have
\begin{equation}\label{eq:Piherrorestimate}
\sum_{i=0}^s\sum_{K\in\mathcal{T}_h}h_K^{2i-2s}|\boldsymbol{\tau}-\Pi_h\boldsymbol{\tau}|_{i,K}^2\lesssim |\boldsymbol{\tau}|_s^2\quad\forall~\boldsymbol{\tau}\in H^s(\Omega;\mathbb S) \textrm{ with } 1\leq s\leq 3,
\end{equation}
\begin{equation}\label{eq:Pihbcommutatediv}
\div(\Pi_h\boldsymbol{\tau})=Q_h(\div\boldsymbol{\tau}) \quad\forall~\boldsymbol{\tau}\in{H}^{1}(\Omega; \mathbb{S}).
\end{equation}
\end{lemma}
\begin{proof}
Since $\boldsymbol{\tau}-\Pi_h\boldsymbol{\tau}=\boldsymbol{\tau}-I_h\boldsymbol{\tau}-I_h^b(\boldsymbol{\tau}-I_h\boldsymbol{\tau})$, we get from the inverse inequality and \eqref{eq:Ihbboundedness} that
\begin{align*}	
|\boldsymbol{\tau}-\Pi_h\boldsymbol{\tau}|_{i,K}&\lesssim |\boldsymbol{\tau}-I_h\boldsymbol{\tau}|_{i,K}+h_K^{-i}\|I_h^b(\boldsymbol{\tau}-I_h\boldsymbol{\tau})\|_{0,K} \\
&\lesssim |\boldsymbol{\tau}-I_h\boldsymbol{\tau}|_{i,K}+h_K^{1-i}\|\div(\boldsymbol{\tau}-I_h\boldsymbol{\tau})\|_{0,K}.
\end{align*}
Then the estimate \eqref{eq:Piherrorestimate} follows from the estimate \eqref{eq:Iherrorestimate}.

Next, we prove \eqref{eq:Pihbcommutatediv}. By \eqref{eq:Ihbcommutatediv}, we have for $K\in\mathcal T_h$ and $\boldsymbol{v}\in\mathbb{P}_1(K;\mathbb{R}^d)/\mathbf{RM}$ that
$$
(\div(\boldsymbol{\tau}-\Pi_h\boldsymbol{\tau}), \boldsymbol{v})_K=(\div(\boldsymbol{\tau}-I_h\boldsymbol{\tau}-I_h^b(\boldsymbol{\tau}-I_h\boldsymbol{\tau})), \boldsymbol{v})_K=0.
$$
For $\boldsymbol{v}\in\mathbf{RM}$, using the integration by parts and \eqref{IhcommuRM},
$$
(\div(\boldsymbol{\tau}-\Pi_h\boldsymbol{\tau}), \boldsymbol{v})_K=(\div(\boldsymbol{\tau}-I_h\boldsymbol{\tau}), \boldsymbol{v})_K=0.
$$
Hence
$$
(\div(\boldsymbol{\tau}-\Pi_h\boldsymbol{\tau}), \boldsymbol{v})_K=0\quad\forall~\boldsymbol{v}\in{\mathbb{P}}_1(K;\mathbb{R}^d), K\in\mathcal T_h,
$$
which implies \eqref{eq:Pihbcommutatediv}.
\end{proof}

Combine \eqref{eq:divsigmah}, \eqref{eq:Pihbcommutatediv} and $\div\boldsymbol{\sigma}=-\boldsymbol{f}$ to produce
\begin{equation*}
\div\boldsymbol{\sigma}_{h}=-Q_h\boldsymbol{f}=\div(\Pi_h\boldsymbol{\sigma}).
\end{equation*}

We are in the position to show the discrete inf-sup condition.
\begin{lemma}
For $\boldsymbol{v}_h \in \boldsymbol{V}_h$, it holds the discrete inf-sup condition
\begin{align}\label{infsupcondition}
\left\|\boldsymbol{v}_h\right\|_{0} \lesssim \sup _{\boldsymbol{\tau}_h \in\boldsymbol{\Sigma}_h} \frac{b(\boldsymbol{\tau}_h, \boldsymbol{v}_h)}{\|\boldsymbol{\tau}_h\|_{{H}(\div)}}.
\end{align}
\end{lemma}	
\begin{proof}
Notice that there exists $\boldsymbol{\tau}\in{H}^{1}(\Omega; \mathbb{S})$ \cite{ArnoldHu2021} such that 
$$
\div\boldsymbol{\tau}=\boldsymbol{v}_h,\quad\|\boldsymbol{\tau}\|_1\lesssim \|\boldsymbol{v}_h\|_0.
$$
It follows from \eqref{eq:Piherrorestimate} and \eqref{eq:Pihbcommutatediv} that
$$
\div(\Pi_h\boldsymbol{\tau})=Q_h(\div\boldsymbol{\tau})=\boldsymbol{v}_h,
\quad
\|\Pi_h\boldsymbol{\tau}\|_{{H}(\div)}\lesssim \|\boldsymbol{\tau}\|_1\lesssim \|\boldsymbol{v}_h\|_0.
$$
This proves the discrete inf-sup condition \eqref{infsupcondition}.
\end{proof}

We will derive another discrete inf-sup condition for the linear form $b(\cdot, \cdot)$. To
this end, introduce mesh dependent norms
	\begin{align*}
		\|\boldsymbol{\tau}\|_{0, h}^2 & :=\|\boldsymbol{\tau}\|_0^2+\sum_{F \in \Delta_{d-1}(\mathcal{T}_h)} h_F\|\boldsymbol{\tau} \boldsymbol{n}\|_{0, F}^2,  \quad
|\boldsymbol{v}|_{1, h}^2  :=\|\boldsymbol{\varepsilon}_h(\boldsymbol{v})\|_0^2+\sum_{F \in \Delta_{d-1}(\mathcal{T}_h)} h_F^{-1}\|\llbracket\boldsymbol{v}\rrbracket\|_{0, F}^2
	\end{align*}
for $\boldsymbol{\tau} \in H^1(\mathcal{T}_h;\mathbb S)$ and $\boldsymbol{v}\in H^1(\mathcal{T}_h;\mathbb R^d)$.
Here $\boldsymbol{\varepsilon}_h$ is the elementwise symmetric gradient. Clearly we have
$$
a(\boldsymbol{\sigma}, \boldsymbol{\tau})\leq \|\boldsymbol{\sigma}\|_{0,h}\|\boldsymbol{\tau}\|_{0,h} \quad \forall~\boldsymbol{\sigma}, \boldsymbol{\tau} \in  H^1(\mathcal{T}_h;\mathbb S).
$$
Apply the integration by parts to get
\begin{equation}\label{bilinearformbrock}
b(\boldsymbol{\tau}, \boldsymbol{v})=-\sum_{K\in\mathcal{T}_h}(\boldsymbol{\tau}, \boldsymbol{\varepsilon}(\boldsymbol{v}))_K+\sum_{F \in \Delta_{d-1}(\mathcal{T}_h)}(\boldsymbol{\tau}\boldsymbol{n}, \llbracket\boldsymbol{v}\rrbracket)_{F}
\end{equation}
for $\boldsymbol{\tau}\in  H^1(\mathcal{T}_h;\mathbb S)\cap{H}(\operatorname{div}, \Omega ; \mathbb{S})$ and $\boldsymbol{v} \in H^1(\mathcal{T}_h;\mathbb R^d)$.
Thus we have
$$
b(\boldsymbol{\tau},\boldsymbol{v})\leq \|\boldsymbol{\tau}\|_{0,h}|\boldsymbol{v}|_{1,h} \quad \forall~\boldsymbol{\tau}\in  H^1(\mathcal{T}_h;\mathbb S)\cap{H}(\operatorname{div}, \Omega ; \mathbb{S}), \boldsymbol{v} \in H^1(\mathcal{T}_h;\mathbb R^d).
$$


\begin{lemma}
For $\boldsymbol{v}_h \in \boldsymbol{V}_h$, it holds the discrete inf-sup condition
		\begin{align}\label{infsupconditionhnorm}
			\left|\boldsymbol{v}_h\right|_{1,h} \lesssim \sup _{\boldsymbol{\tau}_h \in\boldsymbol{\Sigma}_h} \frac{b(\boldsymbol{\tau}_h, \boldsymbol{v}_h)}{\left\|\boldsymbol{\tau}_h\right\|_{0,h}}.
		\end{align}
\end{lemma}
\begin{proof}
Let $\boldsymbol{\tau}_1 \in \boldsymbol{\Sigma}_h$ satisfy that all the DoFs \eqref{lowdivShddof} (DoFs \eqref{sigmadof} for $d=2$) vanish except
\begin{align*}
(\boldsymbol{n}^{\intercal}\boldsymbol{\tau}_1 \boldsymbol{n}, q)_e&=\frac{1}{h_F} (\llbracket \boldsymbol{v}_h\cdot\boldsymbol{n} \rrbracket, q)_F \quad\;\; \forall~q\in \mathbb{P}_{1}(F), F\in \Delta_{d-1}(\mathcal{T}_h),\\
(\boldsymbol{\tau}_1, \boldsymbol{q})_K&=-(\boldsymbol{\varepsilon}(\boldsymbol{v}_h), \boldsymbol{q})_K \quad\quad\;\;\;\,  \forall~\boldsymbol{q} \in \mathbb{P}_{0}(K ; \mathbb{S}), K \in \mathcal{T}_h.
\end{align*}
Let $\boldsymbol{\tau}_2 \in \boldsymbol{\Sigma}_h$ satisfy that all the DoFs \eqref{lowdivShddof} (DoFs \eqref{sigmadof} for $d=2$) vanish except
$$
(\Pi_F\boldsymbol{\tau}_2 \boldsymbol{n}, \boldsymbol{q})_e=\frac{1}{h_F} (\llbracket \Pi_F\boldsymbol{v}_h \rrbracket, \boldsymbol{q})_F \quad\;\; \forall~\boldsymbol q\in {\rm ND}_{0}(F), F\in \Delta_{d-1}(\mathcal{T}_h).
$$
By the norm equivalence,
\begin{equation}\label{eq:202303250}	
\|\boldsymbol{\tau}_1\|_{0,K}\lesssim \|\boldsymbol{\varepsilon}(\boldsymbol{v}_h)\|_{0,K} + \sum_{F\in\partial K}h_F^{-1/2}\|\llbracket \boldsymbol{v}_h\cdot\boldsymbol{n} \rrbracket\|_{0,F}\quad\forall~K\in\mathcal{T}_h,
\end{equation}
\begin{equation}\label{eq:20230325}	
\|\boldsymbol{\tau}_2\|_{0,K}\lesssim \sum_{F\in\partial K}h_F^{-1/2}\|\llbracket \Pi_F\boldsymbol{v}_h\rrbracket\|_{0,F}\quad\forall~K\in\mathcal{T}_h.
\end{equation}
Notice that $(\Pi_F\boldsymbol{\tau}_1 \boldsymbol{n})|_F=0$ and $(\boldsymbol{n}^{\intercal}\boldsymbol{\tau}_2 \boldsymbol{n})|_F=0$ for each edge $F\in\Delta_{d-1}(\mathcal{T}_h)$. We get from~\eqref{bilinearformbrock} that
\begin{equation}\label{eq:202303251}	
b(\boldsymbol{\tau}_1,\boldsymbol{v}_h)=\|\boldsymbol{\varepsilon}_h(\boldsymbol{v}_h)\|_0^2 + \sum_{F \in \Delta_{d-1}(\mathcal{T}_h)} h_F^{-1}\|\llbracket\boldsymbol{v}_h\cdot\boldsymbol{n}\rrbracket\|_{0, F}^2,
\end{equation}
\begin{align*}	
&\quad\, b(\boldsymbol{\tau}_2,\boldsymbol{v}_h)\\
&=\sum_{F \in \Delta_{d-1}(\mathcal{T}_h)}(\Pi_F\boldsymbol{\tau}_2 \boldsymbol{n}, \llbracket\Pi_F\boldsymbol{v}_h\rrbracket-Q_F^{\rm ND}\llbracket\Pi_F\boldsymbol{v}_h\rrbracket)_F + \sum_{F \in\Delta_{d-1}(\mathcal{T}_h)} h_F^{-1}\|Q_F^{\rm ND}\llbracket\Pi_F\boldsymbol{v}_h\rrbracket\|_{0, F}^2 \\
&=\sum_{F \in \Delta_{d-1}(\mathcal{T}_h)} h_F^{-1}\|\llbracket\Pi_F\boldsymbol{v}_h\rrbracket\|_{0, F}^2+\sum_{F \in \Delta_{d-1}(\mathcal{T}_h)}(\Pi_F\boldsymbol{\tau}_2 \boldsymbol{n}, \llbracket\Pi_F\boldsymbol{v}_h\rrbracket-Q_F^{\rm ND}\llbracket\Pi_F\boldsymbol{v}_h\rrbracket)_F \\
&\quad - \sum_{F \in \Delta_{d-1}(\mathcal{T}_h)} h_F^{-1}\big\|\llbracket\Pi_F\boldsymbol{v}_h\rrbracket-Q_F^{\rm ND}\llbracket\Pi_F\boldsymbol{v}_h\rrbracket\big\|_{0, F}^2,
\end{align*}
where $Q_F^{\rm ND}$ is the $L^2$-orthogonal projector onto ${\rm ND}_{0}(F)$ with respect to the inner product $(\cdot,\cdot)_F$.
Thanks to the Cauchy-Schwarz inequality, the inverse trace inequality for polynomials, \eqref{eq:20230325} and the Young's inequality,
\begin{align*}
&\quad-\sum_{F \in \Delta_{d-1}(\mathcal{T}_h)}(\Pi_F\boldsymbol{\tau}_2 \boldsymbol{n}, \llbracket\Pi_F\boldsymbol{v}_h\rrbracket-Q_F^{\rm ND}\llbracket\Pi_F\boldsymbol{v}_h\rrbracket)_F\\
&\leq \sum_{F \in \Delta_{d-1}(\mathcal{T}_h)}\|\Pi_F\boldsymbol{\tau}_2 \boldsymbol{n}\|_{0,F}\| \llbracket\Pi_F\boldsymbol{v}_h\rrbracket-Q_F^{\rm ND}\llbracket\Pi_F\boldsymbol{v}_h\rrbracket\|_{0,F}\\
&\leq C \|\boldsymbol{\tau}_2\|_0\Bigg(\sum_{F \in \Delta_{d-1}(\mathcal{T}_h)} h_F^{-1}\big\|\llbracket\Pi_F\boldsymbol{v}_h\rrbracket-Q_F^{\rm ND}\llbracket\Pi_F\boldsymbol{v}_h\rrbracket\big\|_{0, F}^2\Bigg)^{1/2} \\
&\leq C \Bigg(\sum_{F \in \Delta_{d-1}(\mathcal{T}_h)} h_F^{-1}\big\|\llbracket\Pi_F\boldsymbol{v}_h\rrbracket\big\|_{0, F}^2\Bigg)^{1/2}\Bigg(\sum_{F \in \Delta_{d-1}(\mathcal{T}_h)} h_F^{-1}\big\|\llbracket\Pi_F\boldsymbol{v}_h\rrbracket-Q_F^{\rm ND}\llbracket\Pi_F\boldsymbol{v}_h\rrbracket\big\|_{0, F}^2\Bigg)^{1/2} \\
&\leq\frac{1}{2}\sum_{F \in \Delta_{d-1}(\mathcal{T}_h)} h_F^{-1}\big\|\llbracket\Pi_F\boldsymbol{v}_h\rrbracket\big\|_{0, F}^2 + (C_1-1)\sum_{F \in \Delta_{d-1}(\mathcal{T}_h)} h_F^{-1}\big\|\llbracket\Pi_F\boldsymbol{v}_h\rrbracket-Q_F^{\rm ND}\llbracket\Pi_F\boldsymbol{v}_h\rrbracket\big\|_{0, F}^2,
\end{align*}
where $C_1=1+\frac{1}{2}C^2$ is independent of the mesh size $h$ and the parameter $\lambda$, but may depend on the chunkiness parameter, i.e. the aspect ratio. Hence,
$$
b(\boldsymbol{\tau}_2,\boldsymbol{v}_h)\geq \frac{1}{2}\sum_{F \in \Delta_{d-1}(\mathcal{T}_h)} h_F^{-1}\|\llbracket\Pi_F\boldsymbol{v}_h\rrbracket\|_{0, F}^2 - C_1\sum_{F \in \Delta_{d-1}(\mathcal{T}_h)} h_F^{-1}\big\|\llbracket\Pi_F\boldsymbol{v}_h\rrbracket-Q_F^{\rm ND}\llbracket\Pi_F\boldsymbol{v}_h\rrbracket\big\|_{0, F}^2.
$$
Let $\pi_K$ be the interpolation operator from $H^1(K;\mathbb R^2)$ onto $\mathbf{RM}$ defined by (3.1)-(3.2) in~\cite{Brenner2004}, and let $\pi$ be the elementwise global version of $\pi_K$, i.e., $\pi|_K:=\pi_K$ for each $K\in\mathcal{T}_h$. 
By Lemma~\ref{lem:PiFRM}, $\llbracket\Pi_F(\pi\boldsymbol{v}_h)\rrbracket\in{\rm ND}_{0}(F)$ on each face $F\in \Delta_{d-1}(\mathcal{T}_h)$.
It follows from (3.3)-(3.4) in \cite{Brenner2004} that 
\begin{align*}
&\quad\sum_{F \in \Delta_{d-1}(\mathcal{T}_h)} h_F^{-1}\big\|\llbracket\Pi_F\boldsymbol{v}_h\rrbracket-Q_F^{\rm ND}\llbracket\Pi_F\boldsymbol{v}_h\rrbracket\big\|_{0, F}^2 \\
&= \sum_{F \in \Delta_{d-1}(\mathcal{T}_h)} h_F^{-1}\big\|\llbracket\Pi_F(\boldsymbol{v}_h-\pi\boldsymbol{v}_h)\rrbracket-Q_F^{\rm ND}\llbracket\Pi_F(\boldsymbol{v}_h-\pi\boldsymbol{v}_h)\rrbracket\big\|_{0, F}^2 \\
&\leq \sum_{F \in \Delta_{d-1}(\mathcal{T}_h)} h_F^{-1}\big\|\llbracket\Pi_F(\boldsymbol{v}_h-\pi\boldsymbol{v}_h)\rrbracket\big\|_{0, F}^2\leq C_0\|\boldsymbol{\varepsilon}_h(\boldsymbol{v}_h)\|_0^2.
\end{align*}
Combining the last two inequalities gives
\begin{equation}\label{eq:202303252}	
b(\boldsymbol{\tau}_2,\boldsymbol{v}_h)\geq\frac{1}{2}\sum_{F \in \Delta_{d-1}(\mathcal{T}_h)} h_F^{-1}\big\|\llbracket\Pi_F\boldsymbol{v}_h\rrbracket\big\|_{0, F}^2-C_2\|\boldsymbol{\varepsilon}_h(\boldsymbol{v}_h)\|_0^2
\end{equation}
with $C_2=C_0C_1$.

Take $\boldsymbol{\tau}_h=\boldsymbol{\tau}_1+\frac{1}{2C_2}\boldsymbol{\tau}_2\in\boldsymbol{\Sigma}_h$. We have from \eqref{eq:202303250}-\eqref{eq:202303252} that
$$
b(\boldsymbol{\tau}_h,\boldsymbol{v}_h)\geq\frac{1}{2}\|\boldsymbol{\varepsilon}_h(\boldsymbol{v}_h)\|_0^2 + \sum_{F \in \Delta_{d-1}(\mathcal{T}_h)} h_F^{-1}\|\llbracket\boldsymbol{v}_h\cdot\boldsymbol{n}\rrbracket\|_{0, F}^2 +\frac{1}{4C_2}\sum_{F \in \Delta_{d-1}(\mathcal{T}_h)} h_F^{-1}\|\llbracket\Pi_F\boldsymbol{v}_h\rrbracket\|_{0, F}^2,
$$
$$
\|\boldsymbol{\tau}_h\|_{0,h}\leq \|\boldsymbol{\tau}_1\|_{0,h}+\|\boldsymbol{\tau}_2\|_{0,h}\lesssim|\boldsymbol{v}|_{1, h}.
$$
Therefore the discrete inf-sup condition \eqref{infsupconditionhnorm} is true.
\end{proof}

Recall the coercivity of the bilinear form $a(\cdot, \cdot)$ on the null space of the div operator.
For $\boldsymbol{\tau} \in {H}(\div, \Omega; \mathbb{S})$ satisfying 
$\int_\Omega \tr\boldsymbol{\tau} \,{\rm d}x = 0$ and $\div\boldsymbol{\tau}=0$, we have (cf. \cite[Proposition 9.1.1]{BoffiBrezziFortin2013} and \cite[Section 3.3]{ChenHuHuang2018})
\begin{equation}\label{discretecoercivity} 
\|\boldsymbol{\tau}\|_0^2 \lesssim a(\boldsymbol{\tau}, \boldsymbol{\tau}).
\end{equation}

		
	Applying the Babu{\v{s}}ka-Brezzi theory \cite{BoffiBrezziFortin2013}, from the discrete inf-sup conditions~\eqref{infsupcondition} and (\ref{infsupconditionhnorm}), and the coercivity (\ref{discretecoercivity}), we achieve the well-posedness of the mixed finite element method (\ref{distype1})-(\ref{distype2}), and the following discrete stability results.	
	\begin{theorem}
		The mixed finite element method \eqref{distype1}-\eqref{distype2} is well-posed. We have the discrete stability results
		\begin{align}	\label{eq:discretestability}
			\quad\|\widetilde{\boldsymbol{\sigma}}_h\|_{H(\div)}+\|\widetilde{\boldsymbol{u}}_h\|_0
			\lesssim \sup_{\boldsymbol{\tau}_h\in\boldsymbol{\Sigma}_h,\boldsymbol{v}_h\in \boldsymbol{V}_h}\frac{a(\widetilde{\boldsymbol{\sigma}}_h,\boldsymbol{\tau}_h)+b(\boldsymbol{\tau}_h,\widetilde{\boldsymbol{u}}_h)+b(\widetilde{\boldsymbol{\sigma}}_h, \boldsymbol{v}_h)}{\|\boldsymbol{\tau}_h\|_{H(\div)}+\|\boldsymbol{v}_h\|_0},
		\end{align}
		\begin{align}	\label{eq:discretestabilityhnorm}
			\quad\|\widetilde{\boldsymbol{\sigma}}_h\|_{0,h}+|\widetilde{\boldsymbol{u}}_h|_{1,h}
			\lesssim \sup_{\boldsymbol{\tau}_h\in\boldsymbol{\Sigma}_h,\boldsymbol{v}_h\in \boldsymbol{V}_h}\frac{a(\widetilde{\boldsymbol{\sigma}}_h,\boldsymbol{\tau}_h)+b(\boldsymbol{\tau}_h,\widetilde{\boldsymbol{u}}_h)+b(\widetilde{\boldsymbol{\sigma}}_h, \boldsymbol{v}_h)}{\|\boldsymbol{\tau}_h\|_{0,h}+|\boldsymbol{v}_h|_{1,h}},
		\end{align}
		for any $\widetilde{\boldsymbol{\sigma}}_h\in \boldsymbol{\Sigma}_h, \widetilde{\boldsymbol{u}}_h\in \boldsymbol{V}_h$.
All the hidden constants in these inequalities are independent of $h$ and the parameter $\lambda$.
	\end{theorem}

	\subsection{Error analysis}\label{sec33}
We will show the error analysis of the mixed finite element method \eqref{distype1}-\eqref{distype2}.

\begin{theorem}
Let $(\boldsymbol{\sigma}, \boldsymbol{u})$ and $(\boldsymbol{\sigma}_h, \boldsymbol{u}_h)$ be the solution of the mixed formulation~\eqref{HJ1}-\eqref{HJ2} and the mixed finite element method \eqref{distype1}-\eqref{distype2} respectively. Assume $\boldsymbol{\sigma}\in{H}^3(\Omega;\mathbb{S})$ and $\boldsymbol{u}\in {H}^2(\Omega;\mathbb R^d)$. Then
\begin{align}
\label{eq:errorestimate}
\quad\|\boldsymbol{\sigma}-\boldsymbol{\sigma}_h\|_{0,h} + |Q_h \boldsymbol{u}-\boldsymbol{u}_h|_{1,h}\lesssim h^3\|\boldsymbol{\sigma}\|_{3},\\
\label{erorestimatesigmahdiv}
\|\boldsymbol{\sigma}-\boldsymbol{\sigma}_h\|_{{H}(\div)}\lesssim h^2\|\boldsymbol{\sigma}\|_{3},\\
\label{erorestimateuh0}
\|\boldsymbol{u}-\boldsymbol{u}_h\|_{0}\lesssim h^2(h\|\boldsymbol{\sigma}\|_{3}+\|\boldsymbol{u}\|_2).
\end{align}
All the hidden constants in these inequalities are independent of $h$ and the parameter $\lambda$.
\end{theorem}
\begin{proof}
Subtract \eqref{distype1}-\eqref{distype2} from \eqref{HJ1}-\eqref{HJ2} to obtain the error equation
\begin{equation}\label{eq:error}
a(\boldsymbol{\sigma}-\boldsymbol{\sigma}_h, \boldsymbol{\tau}_h)+b(\boldsymbol{\tau}_h, \boldsymbol{u}-\boldsymbol{u}_h) + b(\boldsymbol{\sigma}-\boldsymbol{\sigma}_h, \boldsymbol{v}_h)=0\quad\forall~\boldsymbol{\tau}_h \in \boldsymbol{\Sigma}_h, \boldsymbol{v}_h \in \boldsymbol{V}_h.
\end{equation}
Adopting \eqref{buortho} and \eqref{eq:Pihbcommutatediv}, we acquire for $\boldsymbol{\tau}_h \in \boldsymbol{\Sigma}_h$ and $\boldsymbol{v}_h \in \boldsymbol{V}_h$ that 
\begin{equation}\label{eq:error1}
a(\boldsymbol{\sigma}-\boldsymbol{\sigma}_h, \boldsymbol{\tau}_h)+b(\boldsymbol{\tau}_h, Q_h\boldsymbol{u}-\boldsymbol{u}_h) + b(\Pi_h\boldsymbol{\sigma}-\boldsymbol{\sigma}_h, \boldsymbol{v}_h)=0.
\end{equation}
Then taking $\widetilde{\boldsymbol{\sigma}}_h = \Pi_h \boldsymbol{\sigma}-\boldsymbol{\sigma}_h$ and $\widetilde{\boldsymbol{u}}_h = Q_h\boldsymbol{u}-\boldsymbol{u}_h$ in \eqref{eq:discretestability} and \eqref{eq:discretestabilityhnorm}, we have
$$
\|\Pi_h \boldsymbol{\sigma}-\boldsymbol{\sigma}_h\|_{{H}(\div)}+\|Q_h \boldsymbol{u}-\boldsymbol{u}_h\|_0
\lesssim \sup _{\boldsymbol{\tau}_h \in \boldsymbol{\Sigma}_h, \boldsymbol{v}_h \in \boldsymbol{V}_h} \frac{a(\Pi_h \boldsymbol{\sigma}-\boldsymbol{\sigma}, \boldsymbol{\tau}_h)}{\|\boldsymbol{\tau}_h\|_{{H}(\div)}+\|\boldsymbol{v}_h\|_0}\lesssim \|\boldsymbol{\sigma}-\Pi_h\boldsymbol{\sigma}\|_0,
$$
$$
\|\Pi_h \boldsymbol{\sigma}-\boldsymbol{\sigma}_h\|_{0,h}+|Q_h \boldsymbol{u}-\boldsymbol{u}_h|_{1,h}
\lesssim \sup _{\boldsymbol{\tau}_h \in \boldsymbol{\Sigma}_h, \boldsymbol{v}_h \in \boldsymbol{V}_h} \frac{a(\Pi_h \boldsymbol{\sigma}-\boldsymbol{\sigma}, \boldsymbol{\tau}_h)}{\|\boldsymbol{\tau}_h\|_{0,h}+|\boldsymbol{v}_h|_{1,h}}\lesssim \|\boldsymbol{\sigma}-\Pi_h\boldsymbol{\sigma}\|_0.
$$
Hence, the estimates \eqref{eq:errorestimate} and \eqref{erorestimatesigmahdiv} follow from the triangle inequality and~\eqref{eq:Piherrorestimate}.
Finally, the estimate \eqref{erorestimateuh0} holds from the triangle inequality, \eqref{eq:Piherrorestimate} and the error estimate of $Q_h$.
\end{proof}

\begin{remark}
The convergence rates of $\|\boldsymbol{u}-\boldsymbol{u}\|_0$, $\|\boldsymbol{\sigma}-\boldsymbol{\sigma}_h\|_0$ and $\|\boldsymbol{\sigma}-\boldsymbol{\sigma}_h\|_{{H}(\div)}$ are optimal. The estimate of $|Q_h \boldsymbol{u}-\boldsymbol{u}_h|_{1,h}$ in \eqref{eq:errorestimate} is superconvergent, which is two-order higher than the optimal one.
\end{remark}

By using the duality argument \cite{DouglasRoberts1985,Stenberg1991}, we can derive the superconvergent estimate of $\|Q_h \boldsymbol{u}-\boldsymbol{u}_h\|_0$. Introduce the dual problem
\begin{equation}\label{linearelasdual}
	\left\{\begin{aligned}
			\mathcal{A} \widetilde{\boldsymbol{\sigma}}-\boldsymbol{\varepsilon}(\widetilde{\boldsymbol{u}})&=0 \qquad\qquad\;\;\; \text { in } \Omega, \\
			\operatorname{div} \widetilde{\boldsymbol{\sigma}}&=Q_h \boldsymbol{u}-\boldsymbol{u}_h \quad \text { in } \Omega, \\
			\widetilde{\boldsymbol{u}}&=0  \qquad\qquad\quad\text { on } \partial \Omega.
		\end{aligned}\right.
	\end{equation}
Under the assumption that $\Omega$ is convex,
it holds the regularity \cite{BrennerSung1992}
\begin{equation}\label{eq:u0regularity}
\|\widetilde{\boldsymbol{\sigma}}\|_1+\|\widetilde{\boldsymbol{u}}\|_2 \lesssim \|Q_h \boldsymbol{u}-\boldsymbol{u}_h\|_0.
\end{equation}

\begin{theorem}\label{thm:errorQhuuL2}
Let $(\boldsymbol{\sigma}, \boldsymbol{u})$ and $(\boldsymbol{\sigma}_h, \boldsymbol{u}_h)$ be the solution of the mixed formulation~\eqref{HJ1}-\eqref{HJ2} and the mixed finite element method \eqref{distype1}-\eqref{distype2} respectively. Assume $\Omega$ is convex, $\boldsymbol{\sigma}\in{H}^3(\Omega;\mathbb{S})$ and $\boldsymbol{u}\in {H}^2(\Omega;\mathbb R^d)$. Then
\begin{equation}
\label{errorQhuuL2}
\|Q_h \boldsymbol{u}-\boldsymbol{u}_h\|_0 \lesssim h^{4}\|\boldsymbol{\sigma}\|_{3}.
\end{equation}
The hidden constant is independent of $h$ and the parameter $\lambda$.
\end{theorem}
\begin{proof}
By the second equation in \eqref{linearelasdual},
\eqref{eq:Pihbcommutatediv}, and \eqref{eq:error1} with $\boldsymbol{\tau}_h=\Pi_h\widetilde{\boldsymbol{\sigma}}$ and $\boldsymbol{v}_h=0$, we have
\[
\|Q_h \boldsymbol{u}-\boldsymbol{u}_h\|_0^2=b(\widetilde{\boldsymbol{\sigma}},Q_h \boldsymbol{u}-\boldsymbol{u}_h)=b(\Pi_h\widetilde{\boldsymbol{\sigma}},Q_h \boldsymbol{u}-\boldsymbol{u}_h)=-a(\boldsymbol{\sigma}-\boldsymbol{\sigma}_h, \Pi_h\widetilde{\boldsymbol{\sigma}}).
\]
Apply the first equation in \eqref{linearelasdual} and
\eqref{eq:error} with $\boldsymbol{\tau}_h=0$ and $\boldsymbol{v}_h=Q_h\widetilde{\boldsymbol{u}}$ to get
\begin{align*}	
\|Q_h \boldsymbol{u}-\boldsymbol{u}_h\|_0^2&=a(\boldsymbol{\sigma}-\boldsymbol{\sigma}_h, \widetilde{\boldsymbol{\sigma}}-\Pi_h\widetilde{\boldsymbol{\sigma}})-a(\boldsymbol{\sigma}-\boldsymbol{\sigma}_h, \widetilde{\boldsymbol{\sigma}}) \\
&=a(\boldsymbol{\sigma}-\boldsymbol{\sigma}_h, \widetilde{\boldsymbol{\sigma}}-\Pi_h\widetilde{\boldsymbol{\sigma}})+b(\boldsymbol{\sigma}-\boldsymbol{\sigma}_h, \widetilde{\boldsymbol{u}}) \\
&=a(\boldsymbol{\sigma}-\boldsymbol{\sigma}_h, \widetilde{\boldsymbol{\sigma}}-\Pi_h\widetilde{\boldsymbol{\sigma}})+b(\boldsymbol{\sigma}-\boldsymbol{\sigma}_h, \widetilde{\boldsymbol{u}}-Q_h\widetilde{\boldsymbol{u}}),
\end{align*}
which combined with the Cauchy-Schwarz inequality,
\eqref{eq:Piherrorestimate}, the error estimate of $Q_h$, and the regularity \eqref{eq:u0regularity} yields
$$
\|Q_h \boldsymbol{u}-\boldsymbol{u}_h\|_0\lesssim h\|\boldsymbol{\sigma}-\boldsymbol{\sigma}_h\|_0 + h^2\|\div(\boldsymbol{\sigma}-\boldsymbol{\sigma}_h)\|_0.
$$
Therefore, the estimate \eqref{errorQhuuL2} holds from the estimates \eqref{eq:errorestimate} and \eqref{erorestimatesigmahdiv}.
\end{proof}
	
The superconvergent estimate of $\|Q_h \boldsymbol{u}-\boldsymbol{u}_h\|_{0,h}$ in \eqref{errorQhuuL2} is also two-order higher than the optimal one.

\subsection{A reduced mixed finite element method}
We can also use the spaces $\boldsymbol{\Sigma}^r_{h}$ and $\boldsymbol{V}^r_h$ to discretize the solution of the mixed formulation~\eqref{HJ1}-\eqref{HJ2}. The resulting reduced mixed finite element method is to find $(\boldsymbol{\sigma}_h, \boldsymbol{u}_h) \in \boldsymbol{\Sigma}^r_{h} \times \boldsymbol{V}^r_h $
	such that
	\begin{align}
		\label{reduceddistype1}	a(\boldsymbol{\sigma}_{h}, \boldsymbol{\tau}_{h})+b(\boldsymbol{\tau}_{h},\boldsymbol{u}_h) &=0 \quad\quad\quad\quad\; \forall~\boldsymbol{\tau}_{h} \in \boldsymbol{\Sigma}^r_{h}, \\
		\label{reduceddistype2}	b(\boldsymbol{\sigma}_{h},
		\boldsymbol{v}_h) &=-(\boldsymbol{f}, \boldsymbol{v}_h) \quad \forall~\boldsymbol{v}_{h} \in \boldsymbol{V}^r_h.
	\end{align}
Applying the argument as in Section~\ref{sec32} and Section~\ref{sec33}, the mixed finite element method~\eqref{reduceddistype1}-\eqref{reduceddistype2} is well-posed, and possesses the following optimal convergence.

\begin{theorem}\label{thm:reducedmfemerror}
Let $(\boldsymbol{\sigma}, \boldsymbol{u})$ and $(\boldsymbol{\sigma}_h, \boldsymbol{u}_h)$ be the solution of the mixed formulation~\eqref{HJ1}-\eqref{HJ2} and the mixed finite element method \eqref{reduceddistype1}-\eqref{reduceddistype2} respectively. Assume $\boldsymbol{\sigma}\in{H}^2(\Omega;\mathbb{S})$ and $\boldsymbol{u}\in {H}^1(\Omega;\mathbb R^d)$. Then
\begin{align*}
\quad\|\boldsymbol{\sigma}-\boldsymbol{\sigma}_h\|_{0,h}+\|Q_h^r\boldsymbol{u}-\boldsymbol{u}_h\|_{0} + |Q_h^r \boldsymbol{u}-\boldsymbol{u}_h|_{1,h}\lesssim h^2\|\boldsymbol{\sigma}\|_{2},\\
\|\boldsymbol{\sigma}-\boldsymbol{\sigma}_h\|_{{H}(\div)}\lesssim h\|\boldsymbol{\sigma}\|_{2},\\
\|\boldsymbol{u}-\boldsymbol{u}_h\|_{0}\lesssim h(h\|\boldsymbol{\sigma}\|_{2}+\|\boldsymbol{u}\|_1),
\end{align*}
where $Q_h^r$ is the $L^2$ orthogonal projector from ${L}^2(\Omega; \mathbb{R}^2)$ onto $\boldsymbol{V}_h^r$. All the hidden constants in these inequalities are independent of $h$ and the parameter $\lambda$.
\end{theorem}

\begin{remark}\rm
Following the proof of Theorem~\ref{thm:errorQhuuL2}, we get
\begin{align*}	
\|Q_h^r\boldsymbol{u}-\boldsymbol{u}_h\|_0^2&\lesssim \|\boldsymbol{\sigma}-\boldsymbol{\sigma}_h\|_0\|\widetilde{\boldsymbol{\sigma}}-\Pi_h\widetilde{\boldsymbol{\sigma}}\|_0 + \|\div(\boldsymbol{\sigma}-\boldsymbol{\sigma}_h)\|_0\|\widetilde{\boldsymbol{u}}-Q_h^r\widetilde{\boldsymbol{u}}\|_0 \\
&\lesssim h^3|\boldsymbol{\sigma}-\boldsymbol{\sigma}_h|_2|\widetilde{\boldsymbol{\sigma}}|_1 + h|\div\boldsymbol{\sigma}|_1\|\widetilde{\boldsymbol{u}}-Q_h^r\widetilde{\boldsymbol{u}}\|_0.
\end{align*}
Notice that $\|\widetilde{\boldsymbol{u}}-Q_h^r\widetilde{\boldsymbol{u}}\|_0\lesssim h|\widetilde{\boldsymbol{u}}|_1$ rather than $\|\widetilde{\boldsymbol{u}}-Q_h^r\widetilde{\boldsymbol{u}}\|_0\lesssim h^2|\widetilde{\boldsymbol{u}}|_2$. So we only have a second order convergence rate for $\|Q_h^r\boldsymbol{u}-\boldsymbol{u}_h\|_0$, which is sharp and numerically verified in Section~\ref{sec:numresult}.
\end{remark}

\section{Numerical results}\label{sec:numresult}
	In this section, we will numerically test the mixed finite element methods \eqref{distype1}-\eqref{distype2} and \eqref{reduceddistype1}-\eqref{reduceddistype2} on the unit square $\Omega=(0,1)\times(0,1)$. 
	Set the Lam\'{e} coefficient $\mu=1$. Let the exact
solution 
	$$
	\boldsymbol{u}=\begin{pmatrix}
		\pi \sin^2(\pi x)\sin(\pi y)\cos(\pi y)\\
		-\pi \sin(\pi x)\cos(\pi x)\sin^2(\pi y)\\
	\end{pmatrix}.
	$$
	The load function $\boldsymbol{f}$ is analytically computed from problem \eqref{linearelasticity}.
Uniform grids with different mesh sizes are adopted in the numerical experiments.

Numerical results of errors $\|\boldsymbol{\sigma}-\boldsymbol{\sigma}_h\|_0$, $\|Q_h\boldsymbol{u}-\boldsymbol{u}_h\|_{0}$ and $|Q_h\boldsymbol{u}-\boldsymbol{u}_h|_{1,h}$ of the mixed finite element methods \eqref{distype1}-\eqref{distype2}
with respect to $h$ for $\lambda= 1, 1000, 10^6, \infty$ are shown in Tables~\ref{table:lambda1}-\ref{table:lambdainfty}. From Tables~\ref{table:lambda1}-\ref{table:lambdainfty} we can see that 
$\|\boldsymbol{\sigma}-\boldsymbol{\sigma}_h\|_0=\mathcal{O}(h^3)$, $\|Q_h\boldsymbol{u}-\boldsymbol{u}_h\|_{0}=\mathcal{O}(h^4)$ and $|Q_h\boldsymbol{u}-\boldsymbol{u}_h|_{1,h}=\mathcal{O}(h^3)$ numerically, which agree with the theoretical estimates \eqref{eq:errorestimate} and \eqref{errorQhuuL2}, respectively.
\begin{table}[ht]
			\setlength{\abovecaptionskip}{0.cm}
		\setlength{\belowcaptionskip}{-0.cm}
		\centering
		\caption{Errors $\|\boldsymbol{\boldsymbol{\sigma}}-\boldsymbol{\boldsymbol{\sigma}}_h\|_0$, $\|Q_h\boldsymbol{u}-\boldsymbol{u}_h\|_{0}$ and $|Q_h\boldsymbol{u}-\boldsymbol{u}_h|_{1,h}$ of the mixed finite element method (\ref{distype1})-(\ref{distype2}) for $\lambda=1$.}
\label{table:lambda1}
		\begin{tabular}[t]{ccccccc}
			\toprule
			$h$ & $\|\boldsymbol{\boldsymbol{\sigma}}-\boldsymbol{\boldsymbol{\sigma}}_h\|_0$ & order & $\|Q_h\boldsymbol{u}-\boldsymbol{u}_h\|_{0}$ & order & $|Q_h\boldsymbol{u}-\boldsymbol{u}_h|_{1,h}$ & order \\
			\midrule 
$2^{-1}$ & 2.3008E+00 & $-$ & 1.0529E-01 & $-$ & 5.5880E-01 & $-$ \\
$2^{-2}$ & 4.4165E-01 & 2.38 & 1.3024E-02 & 3.02 & 1.1766E-01 & 2.25 \\
$2^{-3}$ & 7.5474E-02 & 2.55 & 1.1380E-03 & 3.52 & 2.6621E-02 & 2.14 \\
$2^{-4}$ & 1.1379E-02 & 2.73 & 8.5164E-05 & 3.74 & 4.5965E-03 & 2.53 \\
$2^{-5}$ & 1.5375E-03 & 2.89 & 5.7458E-06 & 3.89 & 6.5103E-04 & 2.82 \\
$2^{-6}$ & 1.9794E-04 & 2.96 & 3.6940E-07 & 3.96 & 8.5131E-05 & 2.93 \\
$2^{-7}$ & 2.5034E-05 & 2.98 & 2.3343E-08 & 3.98 & 1.0826E-05 & 2.98 \\
			\bottomrule
		\end{tabular}
\end{table}
\begin{table}[ht]
			\setlength{\abovecaptionskip}{0.cm}
		\setlength{\belowcaptionskip}{-0.cm}
		\centering
		\caption{Errors $\|\boldsymbol{\boldsymbol{\sigma}}-\boldsymbol{\boldsymbol{\sigma}}_h\|_0$, $\|Q_h\boldsymbol{u}-\boldsymbol{u}_h\|_{0}$ and $|Q_h\boldsymbol{u}-\boldsymbol{u}_h|_{1,h}$ of the mixed finite element method (\ref{distype1})-(\ref{distype2}) for $\lambda=1000$.}
\label{table:lambda1000}
		\begin{tabular}[t]{ccccccc}
			\toprule
			$h$ & $\|\boldsymbol{\sigma}-\boldsymbol{\sigma}_h\|_0$ & order & $\|Q_h\boldsymbol{u}-\boldsymbol{u}_h\|_{0}$ & order & $|Q_h\boldsymbol{u}-\boldsymbol{u}_h|_{1,h}$ & order \\
			\midrule 
$2^{-1}$ & 2.3395E+00 & $-$ & 1.0077E-01 & $-$ & 4.6232E-01 & $-$ \\
$2^{-2}$ & 4.5723E-01 & 2.36 & 1.2815E-02 & 2.98 & 9.3933E-02 & 2.30 \\
$2^{-3}$ & 7.6644E-02 & 2.58 & 1.0647E-03 & 3.59 & 1.8562E-02 & 2.34 \\
$2^{-4}$ & 1.1460E-02 & 2.74 & 7.5377E-05 & 3.82 & 3.1191E-03 & 2.57 \\
$2^{-5}$ & 1.5444E-03 & 2.89 & 4.9424E-06 & 3.93 & 4.3713E-04 & 2.83 \\
$2^{-6}$ & 1.9866E-04 & 2.96 & 3.1410E-07 & 3.98 & 5.6902E-05 & 2.94 \\
$2^{-7}$ & 2.5117E-05 & 2.98 & 1.9740E-08 & 3.99 & 7.2213E-06 & 2.98 \\
			\bottomrule
		\end{tabular}
\end{table}
\begin{table}[ht]
			\setlength{\abovecaptionskip}{0.cm}
		\setlength{\belowcaptionskip}{-0.cm}
		\centering
		\caption{Errors $\|\boldsymbol{\boldsymbol{\sigma}}-\boldsymbol{\boldsymbol{\sigma}}_h\|_0$, $\|Q_h\boldsymbol{u}-\boldsymbol{u}_h\|_{0}$ and $|Q_h\boldsymbol{u}-\boldsymbol{u}_h|_{1,h}$ of the mixed finite element method (\ref{distype1})-(\ref{distype2}) for $\lambda=10^6$.}
\label{table:lambda106}
		\begin{tabular}[t]{ccccccc}
			\toprule
			$h$ & $\|\boldsymbol{\sigma}-\boldsymbol{\sigma}_h\|_0$ & order & $\|Q_h\boldsymbol{u}-\boldsymbol{u}_h\|_{0}$ & order & $|Q_h\boldsymbol{u}-\boldsymbol{u}_h|_{1,h}$ & order \\
			\midrule 
$2^{-1}$ & 2.3397E+00 & $-$ & 1.0076E-01 & $-$ & 4.6222E-01 & $-$ \\
$2^{-2}$ & 4.5731E-01 & 2.36 & 1.2816E-02 & 2.97 & 9.3911E-02 & 2.30 \\
$2^{-3}$ & 7.6649E-02 & 2.58 & 1.0647E-03 & 3.59 & 1.8549E-02 & 2.34 \\
$2^{-4}$ & 1.1461E-02 & 2.74 & 7.5367E-05 & 3.82 & 3.1165E-03 & 2.57 \\
$2^{-5}$ & 1.5444E-03 & 2.89 & 4.9415E-06 & 3.93 & 4.3677E-04 & 2.84 \\
$2^{-6}$ & 1.9866E-04 & 2.96 & 3.1404E-07 & 3.98 & 5.6854E-05 & 2.94 \\
$2^{-7}$ & 2.5117E-05 & 2.98 & 1.9735E-08 & 3.99 & 7.2151E-06 & 2.98 \\
			\bottomrule
		\end{tabular}
\end{table}
\begin{table}[htbp]
		\setlength{\abovecaptionskip}{0.cm}
		\setlength{\belowcaptionskip}{-0.cm}
		\centering
		\caption{Errors $\|\boldsymbol{\boldsymbol{\sigma}}-\boldsymbol{\boldsymbol{\sigma}}_h\|_0$, $\|Q_h\boldsymbol{u}-\boldsymbol{u}_h\|_{0}$ and $|Q_h\boldsymbol{u}-\boldsymbol{u}_h|_{1,h}$ of the mixed finite element method (\ref{distype1})-(\ref{distype2}) for $\lambda=\infty$.}
\label{table:lambdainfty}
		\begin{tabular}[t]{ccccccc}
			\toprule
			$h$ & $\|\boldsymbol{\sigma}-\boldsymbol{\sigma}_h\|_0$ & order & $\|Q_h\boldsymbol{u}-\boldsymbol{u}_h\|_{0}$ & order & $|Q_h\boldsymbol{u}-\boldsymbol{u}_h|_{1,h}$ & order \\
			\midrule 
$2^{-1}$ & 2.3397E+00 & $-$ & 1.0076E-01 & $-$ & 4.6222E-01 & $-$ \\
$2^{-2}$ & 4.5731E-01 & 2.36 & 1.2816E-02 & 2.97 & 9.3911E-02 & 2.30 \\
$2^{-3}$ & 7.6649E-02 & 2.58 & 1.0647E-03 & 3.59 & 1.8549E-02 & 2.34 \\
$2^{-4}$ & 1.1461E-02 & 2.74 & 7.5367E-05 & 3.82 & 3.1165E-03 & 2.57 \\
$2^{-5}$ & 1.5444E-03 & 2.89 & 4.9415E-06 & 3.93 & 4.3677E-04 & 2.84 \\
$2^{-6}$ & 1.9866E-04 & 2.96 & 3.1404E-07 & 3.98 & 5.6854E-05 & 2.94 \\
$2^{-7}$ & 2.5117E-05 & 2.98 & 1.9735E-08 & 3.99 & 7.2151E-06 & 2.98 \\
			\bottomrule
		\end{tabular}
\end{table}

Numerical results of errors $\|\boldsymbol{\sigma}-\boldsymbol{\sigma}_h\|_0$, $\|Q_h\boldsymbol{u}-\boldsymbol{u}_h\|_{0}$ and $|Q_h\boldsymbol{u}-\boldsymbol{u}_h|_{1,h}$ of the mixed finite element methods \eqref{reduceddistype1}-\eqref{reduceddistype2}
with respect to $h$ for $\lambda= 1, 1000, 10^6, \infty$ are shown in Tables~\ref{table:reducedlambda1}-\ref{table:reducedlambdainfty}. From Tables~\ref{table:reducedlambda1}-\ref{table:reducedlambdainfty} we can see that 
$\|\boldsymbol{\sigma}-\boldsymbol{\sigma}_h\|_0=\mathcal{O}(h^2)$, $\|Q_h\boldsymbol{u}-\boldsymbol{u}_h\|_{0}=\mathcal{O}(h^2)$ and $|Q_h\boldsymbol{u}-\boldsymbol{u}_h|_{1,h}=\mathcal{O}(h^2)$ numerically, which agree with the theoretical estimates in Theorem~\ref{thm:reducedmfemerror}.
\begin{table}[ht]
			\setlength{\abovecaptionskip}{0.cm}
		\setlength{\belowcaptionskip}{-0.cm}
		\centering
		\caption{Errors $\|\boldsymbol{\boldsymbol{\sigma}}-\boldsymbol{\boldsymbol{\sigma}}_h\|_0$, $\|Q_h^r\boldsymbol{u}-\boldsymbol{u}_h\|_{0}$ and $|Q_h^r\boldsymbol{u}-\boldsymbol{u}_h|_{1,h}$ of the reduced mixed finite element method (\ref{reduceddistype1})-(\ref{reduceddistype2}) for $\lambda=1$.}
\label{table:reducedlambda1}
		\begin{tabular}[t]{ccccccc}
			\toprule
			$h$ & $\|\boldsymbol{\boldsymbol{\sigma}}-\boldsymbol{\boldsymbol{\sigma}}_h\|_0$ & order & $\|Q_h^r\boldsymbol{u}-\boldsymbol{u}_h\|_{0}$ & order & $|Q_h^r\boldsymbol{u}-\boldsymbol{u}_h|_{1,h}$ & order \\
			\midrule 
$2^{-1}$ & 3.0150E+00 & $-$ & 2.0788E-01 & $-$ & 8.5503E-01 & $-$ \\
$2^{-2}$ & 1.0545E+00 & 1.52 & 7.8315E-02 & 1.41 & 3.3356E-01 & 1.36 \\
$2^{-3}$ & 2.6116E-01 & 2.01 & 2.0301E-02 & 1.95 & 8.4282E-02 & 1.98 \\
$2^{-4}$ & 6.4955E-02 & 2.01 & 5.1084E-03 & 1.99 & 2.1035E-02 & 2.00 \\
$2^{-5}$ & 1.6213E-02 & 2.00 & 1.2789E-03 & 2.00 & 5.2619E-03 & 2.00 \\
$2^{-6}$ & 4.0521E-03 & 2.00 & 3.1984E-04 & 2.00 & 1.3166E-03 & 2.00 \\
$2^{-7}$ & 1.0130E-03 & 2.00 & 7.9968E-05 & 2.00 & 3.2933E-04 & 2.00 \\
			\bottomrule
		\end{tabular}
\end{table}
\begin{table}[ht]
			\setlength{\abovecaptionskip}{0.cm}
		\setlength{\belowcaptionskip}{-0.cm}
		\centering
		\caption{Errors $\|\boldsymbol{\boldsymbol{\sigma}}-\boldsymbol{\boldsymbol{\sigma}}_h\|_0$, $\|Q_h^r\boldsymbol{u}-\boldsymbol{u}_h\|_{0}$ and $|Q_h^r\boldsymbol{u}-\boldsymbol{u}_h|_{1,h}$ of the reduced mixed finite element method (\ref{reduceddistype1})-(\ref{reduceddistype2}) for $\lambda=1000$.}
\label{table:reducedlambda1000}
		\begin{tabular}[t]{ccccccc}
			\toprule
			$h$ & $\|\boldsymbol{\sigma}-\boldsymbol{\sigma}_h\|_0$ & order & $\|Q_h^r\boldsymbol{u}-\boldsymbol{u}_h\|_{0}$ & order & $|Q_h^r\boldsymbol{u}-\boldsymbol{u}_h|_{1,h}$ & order \\
			\midrule 
$2^{-1}$ & 3.0554E+00 & $-$ & 2.1197E-01 & $-$ & 8.5721E-01 & $-$ \\
$2^{-2}$ & 1.1000E+00 & 1.47 & 7.9081E-02 & 1.42 & 3.3510E-01 & 1.36 \\
$2^{-3}$ & 2.7584E-01 & 2.00 & 2.0429E-02 & 1.95 & 8.3554E-02 & 2.00 \\
$2^{-4}$ & 6.9016E-02 & 2.00 & 5.1319E-03 & 1.99 & 2.0735E-02 & 2.01 \\
$2^{-5}$ & 1.7263E-02 & 2.00 & 1.2842E-03 & 2.00 & 5.1763E-03 & 2.00 \\
$2^{-6}$ & 4.3170E-03 & 2.00 & 3.2111E-04 & 2.00 & 1.2940E-03 & 2.00 \\
$2^{-7}$ & 1.0794E-03 & 2.00 & 8.0281E-05 & 2.00 & 3.2355E-04 & 2.00 \\
			\bottomrule
		\end{tabular}
\end{table}
\begin{table}[ht]
			\setlength{\abovecaptionskip}{0.cm}
		\setlength{\belowcaptionskip}{-0.cm}
		\centering
		\caption{Errors $\|\boldsymbol{\boldsymbol{\sigma}}-\boldsymbol{\boldsymbol{\sigma}}_h\|_0$, $\|Q_h^r\boldsymbol{u}-\boldsymbol{u}_h\|_{0}$ and $|Q_h^r\boldsymbol{u}-\boldsymbol{u}_h|_{1,h}$ of the reduced mixed finite element method (\ref{reduceddistype1})-(\ref{reduceddistype2}) for $\lambda=10^6$.}
\label{table:reducedlambda106}
		\begin{tabular}[t]{ccccccc}
			\toprule
			$h$ & $\|\boldsymbol{\sigma}-\boldsymbol{\sigma}_h\|_0$ & order & $\|Q_h^r\boldsymbol{u}-\boldsymbol{u}_h\|_{0}$ & order & $|Q_h^r\boldsymbol{u}-\boldsymbol{u}_h|_{1,h}$ & order \\
			\midrule 
$2^{-1}$ & 3.0556E+00 & $-$ & 2.1199E-01 & $-$ & 8.5722E-01 & $-$ \\
$2^{-2}$ & 1.1002E+00 & 1.47 & 7.9084E-02 & 1.42 & 3.3511E-01 & 1.36 \\
$2^{-3}$ & 2.7591E-01 & 2.00 & 2.0430E-02 & 1.95 & 8.3554E-02 & 2.00 \\
$2^{-4}$ & 6.9035E-02 & 2.00 & 5.1321E-03 & 1.99 & 2.0735E-02 & 2.01 \\
$2^{-5}$ & 1.7268E-02 & 2.00 & 1.2842E-03 & 2.00 & 5.1762E-03 & 2.00 \\
$2^{-6}$ & 4.3182E-03 & 2.00 & 3.2112E-04 & 2.00 & 1.2940E-03 & 2.00 \\
$2^{-7}$ & 1.0797E-03 & 2.00 & 8.0284E-05 & 2.00 & 3.2354E-04 & 2.00 \\
			\bottomrule
		\end{tabular}
\end{table}
\begin{table}[htbp]
		\setlength{\abovecaptionskip}{0.cm}
		\setlength{\belowcaptionskip}{-0.cm}
		\centering
		\caption{Errors $\|\boldsymbol{\boldsymbol{\sigma}}-\boldsymbol{\boldsymbol{\sigma}}_h\|_0$, $\|Q_h^r\boldsymbol{u}-\boldsymbol{u}_h\|_{0}$ and $|Q_h^r\boldsymbol{u}-\boldsymbol{u}_h|_{1,h}$ of the reduced mixed finite element method (\ref{reduceddistype1})-(\ref{reduceddistype2}) for $\lambda=\infty$.}
\label{table:reducedlambdainfty}
		\begin{tabular}[t]{ccccccc}
			\toprule
			$h$ & $\|\boldsymbol{\sigma}-\boldsymbol{\sigma}_h\|_0$ & order & $\|Q_h^r\boldsymbol{u}-\boldsymbol{u}_h\|_{0}$ & order & $|Q_h^r\boldsymbol{u}-\boldsymbol{u}_h|_{1,h}$ & order \\
			\midrule 
$2^{-1}$ & 3.0556E+00 & $-$ & 2.1199E-01 & $-$ & 8.5722E-01 & $-$ \\
$2^{-2}$ & 1.1002E+00 & 1.47 & 7.9084E-02 & 1.42 & 3.3511E-01 & 1.36 \\
$2^{-3}$ & 2.7591E-01 & 2.00 & 2.0430E-02 & 1.95 & 8.3554E-02 & 2.00 \\
$2^{-4}$ & 6.9035E-02 & 2.00 & 5.1321E-03 & 1.99 & 2.0735E-02 & 2.01 \\
$2^{-5}$ & 1.7268E-02 & 2.00 & 1.2842E-03 & 2.00 & 5.1762E-03 & 2.00 \\
$2^{-6}$ & 4.3182E-03 & 2.00 & 3.2112E-04 & 2.00 & 1.2940E-03 & 2.00 \\
$2^{-7}$ & 1.0797E-03 & 2.00 & 8.0284E-05 & 2.00 & 3.2354E-04 & 2.00 \\
			\bottomrule
		\end{tabular}
\end{table}

		\section*{Conflict of interest} The authors declare that they have no conflict of interest.
\bibliographystyle{abbrv}
\bibliography{./refs}
\end{document}